\documentclass[a4paper]{article}
\usepackage{geometry}
\usepackage[utf8]{inputenc}
\usepackage[T1]{fontenc}
\usepackage{graphicx}
\usepackage{amsmath}
\usepackage{amssymb}
\usepackage{float}
\usepackage{caption}
\usepackage{hyperref}
\usepackage{algorithm}

\usepackage[noend]{algpseudocode}
\usepackage{xcolor}
\usepackage{mathtools}
\mathtoolsset{showonlyrefs}

\usepackage{amsfonts}
\usepackage{amsthm}
\usepackage{dsfont}
\usepackage{stmaryrd}
\usepackage[algo2e,ruled,vlined]{algorithm2e}

\usepackage[backend=bibtex,maxbibnames=5, style=alphabetic]{biblatex}
\def \trans{^{\scriptscriptstyle{\intercal}}}

\def \ep{\hbox{ }\hfill$\Box$}

\DeclareMathOperator{\Tr}{Tr}

\DeclareMathOperator{\R}{\mathbb{R}}
\def\Lc{{\cal L}}
\def \E{\mathbb{E}}
\def \F{\mathbb{F}}
\def \M{\mathbb{M}}

\def\1{{\bf 1}}

\def \N{\mathbb{N}}
\def\Fc{{\cal F}}

\def\Uc{{\cal U}}

\def\Wc{{\cal W}}
\def\Xc{{\cal X}}
\def\Zc{{\cal Z}}
\def\Nc{{\cal N}}
\def\argmin_#1{\underset{#1}{\mathrm{argmin\, }}}

\newcommand{\di}{\mathrm{d}}

\def \Sum{\displaystyle\sum}
\def \Sum{\displaystyle\sum}

\def\Zc{{\cal Z}}
\def\Gc{{\cal G}}
\def\Fc{{\cal F}}
\def\Uc{{\cal U}}

\def\Lc{{\cal L}}
\def\Sc{{\cal S}}
\def\Xc{{\cal X}}
\def\Nc{{\cal N}}
\def\Wc{{\cal W}}
\def\Jc{{\cal J}}
\def\eps{{\varepsilon}}

\def\bZc{\boldsymbol{\Zc}}

\newcommand{\Pro}{\mathbb{P}}

\def \J{\mathbb{J}}
\def \M{\mathbb{M}}

\def\Lc{{\cal L}}

\def \trans{^{\scriptscriptstyle{\intercal}}}
\def \ep{\hbox{ }\hfill$\Box$}

\newtheorem{Theorem}{Theorem}[section]

\newtheorem{Proposition}{Proposition}[section]
\newtheorem{Assumption}{Assumption}[section]
\newtheorem{Lemma}{Lemma}[section]

\newtheorem{Remark}{Remark}[section]

\numberwithin{equation}{section}

\title{Approximation error analysis of some deep backward schemes for nonlinear PDEs
\thanks{This work is supported by  FiME, Laboratoire de Finance des March\'es de l'Energie, and the ''Finance and Sustainable Development'' EDF - CACIB Chair.}
}

\author{Maximilien \textsc{Germain}
\footnote{EDF R\&D, LPSM, Université de Paris  \sf \href{mailto:Maximilien.Germain at edf.fr}{mgermain at lpsm.paris}} \and  Huyên \textsc{Pham}
\footnote{LPSM, Université de Paris, FiME, CREST ENSAE \sf \href{mailto:pham at lpsm.paris}{pham at lpsm.paris}} \and  Xavier \textsc{Warin}
\footnote{EDF R\&D, FiME \sf \href{mailto:xavier.warin at  edf.fr}{xavier.warin at edf.fr}} 
}

\date{\today\\ {\it to appear in SIAM Journal on Scientific Computing}}

\def \trans{^{\scriptscriptstyle{\intercal}}}
\def \ep{\hbox{ }\hfill$\Box$}
\def \J{\mathbb{J}}
\def \M{\mathbb{M}}

\def \Sum{\displaystyle\sum}

\def\Zc{{\cal Z}}
\def\Gc{{\cal G}}
\def\Fc{{\cal F}}
\def\Uc{{\cal U}}

\def\Xc{{\cal X}}
\def\Nc{{\cal N}}
\def\Wc{{\cal W}}
\def\Jc{{\cal J}}
\def\eps{{\varepsilon}}

\def\bZc{\boldsymbol{\Zc}}

\def\Lc{{\cal L}}

\begin{document}

\maketitle

\begin{abstract}
  Recently proposed numerical algorithms  for solving high-dimensional nonlinear partial differential equations (PDEs) based on neural networks  have shown their remarkable performance.  
  We review some of them and study their convergence properties.  
  The methods rely on  probabilistic representation of PDEs by backward stochastic differential equations (BSDEs) and their iterated time discretization. Our proposed algorithm, called deep backward multistep scheme (MDBDP), is a machine learning version of the LSMDP scheme of Gobet, Turkedjiev (Math. Comp. 85, 2016). It estimates simultaneously by backward induction the solution and its gradient  by neural networks through sequential minimizations of suitable quadratic loss functions that are performed by stochastic gra\-dient descent.  
  Our main theoretical contribution is to provide  an  approximation error analysis of the MDBDP scheme as well as the deep splitting (DS) scheme for semilinear PDEs  designed  in Beck, Becker, Cheridito, Jentzen, Neufeld (2019). We also supplement the error analysis of the DBDP scheme of Huré, Pham, Warin (Math. Comp. 89, 2020). 
  This yields notably convergence rate in terms of the number of neurons for a class of deep Lipschitz continuous GroupSort neural networks when the PDE is linear in the gradient of the solution for the MDBDP scheme, and in the semilinear case for the DBDP scheme. 
  We illustrate our results with some numerical tests that are compared with some other  machine learning algorithms in the literature. 
\end{abstract}

\section{Introduction}
Let us consider the nonlinear parabolic  partial differential equation (PDE)  of the form
\begin{equation}\label{eq: semilinear PDE}
\begin{cases}
\partial_t u + \mu \cdot D_x u + \frac{1}{2} \Tr(\sigma\sigma\trans D^2_x u) \; = \;  f(\cdot,\cdot,u,\sigma\trans D_x u)  &\mathrm{ on } \ [0,T)\times\R^d\\
u(T,\cdot) \; = \;  g  &\mathrm{ on } \ \R^d, 
\end{cases}
\end{equation} 
with $\mu,\sigma$  functions defined on $[0,T]\times\R^d$, valued respectively in $\R^d$, and $\M^d$ (the set of $d\times d$ matrices), 
a nonlinear generator function $f$ defined on $[0,T]\times\R^d\times\R\times\R^d$, and a terminal function  $g$ defined on $\R^d$. Here,  
the operators $D_x, D^2_x$ refer respectively to the first and second order spatial derivatives, the symbol  $.$ denotes the scalar product, and $\trans$ is the transpose  of vector or matrix.     

A major challenge in the numerical resolution of such semilinear PDEs is the so-called "curse of dimensionality" making unfeasible the standard discretization  of the state space in dimension greater than 3. Probabilistic mesh-free methods based on the Backward Stochastic Differential Equation (BSDE) representation of semilinear PDEs through the nonlinear Feynman-Kac formula were developed in \cite{Z04}, \cite{BT04}, \cite{HLOTTW19}, and (ii) on  multilevel Picard methods, developed in  \cite{hutetal18}  with algorithms based on Picard iterations, multi-level techniques and automatic differentiation. These methods permit to handle some 
PDEs with non linearity in $u$ and its gradient $D_x u$, with  convergence results as well as numerous numerical examples showing their efficiency in high dimension.

Over the last few years,  machine learning methods have emerged since the pioneering papers by  \cite{HJE17}  and \cite{SS17}, and have shown their efficiency for solving high-dimensional 
nonlinear PDEs by means of  neural networks approximation. The work \cite{HJE17} introduces a global machine learning resolution technique via a BSDE approach. The solution is represented by one feedforward neural network by time step, whose parameters are chosen as solutions of a single global optimization problem. It allows to solve PDEs in high dimension and a convergence study of Deep BSDE is 
conducted in \cite{HL18}. 
The Deep Galerkin method of \cite{SS17} proposes another global meshfree method with a random sampling of time and space points inside a bounded domain. 
A different point of view  is proposed by \cite{HPW19} with convergence results in $L^2$ for solving semilinear PDEs, where the 
solution and its gradient are estimated simultaneously by backward induction through the minimization of sequential loss functions. Similar idea also 
appears  in \cite{SS17} for linear PDEs.  At the cost of solving multiple optimization problems, the Deep Backward scheme (DBDP) of \cite{HPW19} 
verifies better stability and accuracy properties than the global method in \cite{HJE17}, as illustrated on several test cases.  The recent paper \cite{BBCJN19} also introduces machine learning schemes based on local loss functions, called  Deep Splitting (DS) method which estimates the PDE solution through backward explicit local optimization problems relying on a neural network regression method for the computation of conditional expectations.

In this paper,  we propose machine learning schemes that use  multistep methods introduced in \cite{BD07} and  \cite{GT14}. 
The idea is to rely on the whole previously computed values of the discretized processes in the backward computations of the approximation as it is expected to yield a better propagation of regression errors. 
We shall  develop this approach  to the  DBDP scheme of  \cite{HPW19}, leading to the so-called  deep backward multi-step scheme (MDBDP). This can be viewed as  
a machine learning version of the Multi-step Forward Dynamic Programming method studied by \cite{GT14}. However, instead of solving at each time step two regression problems, our approach allows to consider only a single minimization as in the DBDP scheme. Compared to the latter, the multi-step consideration is expected to provide better accuracy by reducing the propagation of errors in the backward induction.Our main  theoretical contribution is a detailed study of the approximation error of MDBDP scheme, through standard stability-type arguments for BSDEs (see e.g. Section 4.4 in \cite{zhangBook} for the continuous time case). The arguments can be adapted to obtain the convergence of the  DS scheme introduced in \cite{BBCJN19}. Furthermore, by relying on recent approximation results for deep neural networks in \cite{TSB21}, 
we obtain a rate of convergence of our scheme in terms of the number of neurons, and  supplement the convergence analysis  of the DBDP scheme \cite{HPW19}.

We provide some numerical tests of our proposed algorithms, 
which show the benefit of multistep schemes, 
and compare our results with the cited machine learning schemes. Notice that the GroupSort network is used for theoretical analysis but in the numerical implementation, we applied standard networks with tanh as activation function. The theoretical analysis of the convergence of methods relying on standard neural networks is left to future research. More numerical examples and tests  are presented  
in the extended first arXiv version \cite{GPW20} of this paper.

The plan of the paper is the following.  In Section \ref{sec:semischeme}, we give a brief reminder on neural networks and notably on a specific class of deep  network functions considered in \cite{ALG19,TSB21}  that yields an approximation result with rate of convergence for Lipschitz functions. We also review  machine learning schemes for the numerical resolution of semilinear PDEs. We then describe in detail the MDBDP scheme.
We state in Section \ref{sec: convergence}  the convergence of the MDBDP, DS, and DBDP schemes, 
while Section \ref{sec:proof} is devoted to the proof of these results. 
Section \ref{sec: numerics} gives some numerical tests for illustration.

\section{BSDE Machine Learning Schemes for Semilinear PDEs} \label{sec:semischeme}

In this section, we review recent numerical schemes, and present our new scheme for  the resolution of the semi-linear PDE \eqref{eq: semilinear PDE} by approximations in the class of neural networks and relying on probabilistic representation of the solution to the PDE.

\subsection{Neural Networks} \label{sec:NN}

We denote by 
\begin{align*}
\Lc_{d_1,d_2}^{\rho} &=\;  \Big\{ \phi : \R^{d_1} \rightarrow \R^{d_2}: \exists \;  (\Wc,\beta) \in   \R^{d_2\times d_1} \times \R^{d_2}, \; 
\phi(x) \; = \: \rho( \Wc x + \beta) \; \Big\}, 
\end{align*}
the set of layer  functions with input dimension $d_1$, output dimension $d_2$, and activation function $\rho$ $:$ $\R^{d_2}$ $\rightarrow$ $\R^{d_2}$.  Usually, the activation is applied component-wise via a one-dimensional activation function, i.e.,  
$\rho_{}(x_1,\ldots,x_{d_2})$ $=$ $\big(\hat \rho(x_1),\ldots,\hat \rho(x_{d_2})\big)$ with $\hat \rho : \R \mapsto \R$, to the affine  map $x$ $\in$ $\R^{d_1}$ $\mapsto$ $\Wc x + \beta$ $\in$ $\R^{d_2}$, 
with a matrix $\Wc$ called weight, and vector $\beta$ called bias.    
Standard examples of activation functions $\hat\rho$ are  the sigmoid, the ReLU, the $\tanh$. 
When $\rho_{}$ is the identity function, we simply write $\Lc_{d_1,d_2}$.   

We then define 
\begin{align*}
\Nc^\rho_{d_0,d',\ell,m} &= \;  \Big\{ \varphi : \R^{d_0}  \rightarrow \R^{d'}: \exists   \phi_0 \in  \Lc^{\rho_0}_{d_0,m_0}, \; \exists \phi_i \in \Lc^{\rho_i}_{m_{i-1},m_i}, i=1,\ldots,\ell-1, \;   \\%,d'
& \hspace{4cm} \exists  \phi_\ell \in \Lc_{m_{l-1},d'},  \varphi  \; = \;  \phi_\ell \circ \phi_{\ell -1} \circ \cdots \circ \phi_0 \Big\},
\end{align*}
as the set of feedforward neural networks with input layer dimension $d_0$, output layer dimension $d'$, and $\ell$  hidden layers with $m_i$ neurons per layer ($i=0,\cdots,\ell-1$).   
These numbers $d_0,d',\ell$, the sequence $m$ $=$ $(m_i)_{i= 0,\ldots, \ell-1}$, and sequence of activation functions $\rho = (\rho_{i})_{i= 0,\ldots, \ell-1}$,  form the architecture of the network. 
In the sequel, we shall mostly work with the case 
$d_0$ $=$ $d$ (dimension of  the state variable $x$).

A given network function $\varphi$ $\in$ $\Nc^\rho_{d_0,d',\ell,m}$ is determined by the weight/bias parameters $\theta$ $=$ $(\Wc_0,\beta_0,\ldots,\Wc_\ell,\beta_\ell)$ defining the layer functions 
$\phi_0 \ldots,\phi_\ell$, and we shall sometimes write $\varphi$ $=$ $\varphi_\theta$.

\vspace{1mm}

\vspace{1mm}

We recall the fundamental result of \cite{HSW89} that justifies  the use of neural networks as function approximators, in the usual case of activation functions applied componentwise at each hidden layer.  

\vspace{1mm} 

\noindent {\bf Universal approximation theorem.}  
The space 
$\bigcup_{i=0}^{\ell-1} \bigcup_{m_i=0}^\infty \Nc^\rho_{d_0,d',\ell,m}$  is dense in  $L^2(\nu)$, the set of measurable functions $h$ $:$  $\R^{d_0}$ $\rightarrow$ $\R^{d'}$ s.t.  
$\int |h(x)|_{_2}^2 \nu(\di x) < \infty$, for any finite  measure $\nu$ on $\R^{d_0}$, whenever $\rho$ is continuous and non-constant. 

\vspace{3mm}

This universal approximation theorem does not provide any rate of convergence, nor reveals even in theory how to achieve a given accuracy for a fixed number of neurons.   Some results give rates for the approximation of functions in Sobolev spaces \cite{p99}, for bounded convex subdifferentiable Lipschitz functions \cite{b15} or bounded Lipschitz functions \cite{y17}, 
but here, we need  a result related to (possibly unbounded) Lipschitz functions. The paper \cite{B17} provides a possible answer in this direction, but we instead rely on a simpler approach in \cite{TSB21}, building on the GroupSort deep neural networks introduced by \cite{ALG19}. Let $\kappa\in\N^*,\ \kappa\geq 2$,  be a grouping size, dividing the number of neurons $m_i=\kappa n_i$, at each layer  
$i=0,\cdots \ell-1$. $\sum_{i=0}^{\ell-1} m_i$ will be refered to as the width of the network and $\ell+1$ as its depth. The GroupSort networks correspond to classical deep feedforward neural networks in $\Nc^{\zeta_\kappa}_{d,1,\ell,m}$ with a specific sequence of activation function ${\zeta_\kappa}$ $=$ $(\zeta_\kappa^i)_{i=0,\ldots,\ell-1}$, and one-dimensional output. Each nonlinear function ${\zeta_\kappa^i}$ divides its input into groups of size $\kappa$ and sorts each group in decreasing order, see Figure \ref{fig: groupsort}. 
\begin{figure}
    \centering
    \includegraphics[width=10cm]{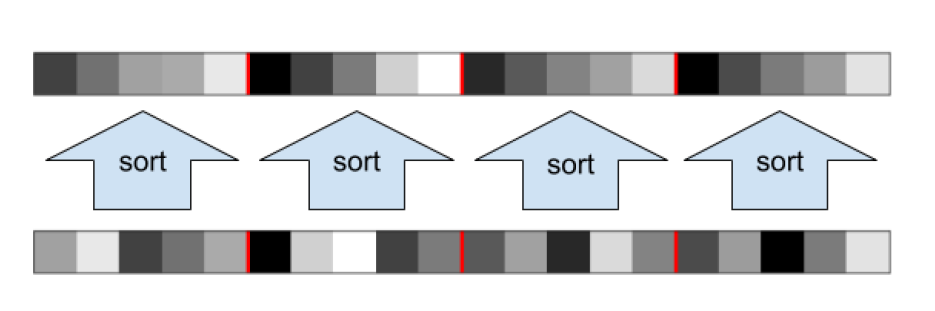}
    \caption{GroupSort activation function ${\zeta_\kappa}$ with grouping size $\kappa=5$ and $m=20$ neurons, figure from \cite{ALG19}.}
    \label{fig: groupsort}
\end{figure}
Moreover, by enforcing the parameters of the GroupSort to satisfy with the Euclidian norm $| \cdot|_2$ and the $\ell_\infty$ norm $|\cdot|_\infty$: 
\begin{equation}\label{eq: bound coeffs}
\sup_{| x|_2 = 1} | \Wc_0 x|_\infty\leq 1,\    \sup_{| x|_\infty = 1} | \Wc_i x|_\infty\leq 1,\ |\beta_j|_\infty \leq M,\ i=1,\cdots,l,\ j=0,\cdots,l
\end{equation} for some $M>0$, the related GroupSort neural networks from $\Nc^{\zeta_\kappa}_{d,d',\ell,m}$ are 1-Lipschitz. 
The space of such 1-Lipschitz GroupSort neural networks is called $\Sc_{d,\ell,m}^{\zeta_\kappa}$:
\begin{align*}
    \Sc_{d,\ell,m}^{\zeta_\kappa} = \{& \varphi_{(\Wc_0,\beta_0,\ldots,\Wc_\ell,\beta_\ell)} \in \Nc^{\zeta_\kappa}_{d,1,\ell,m},\ \sup_{| x|_2 = 1} | \Wc_0 x|_\infty\leq 1,\    \sup_{| x|_\infty = 1} | \Wc_i x|_\infty\leq 1,\\ & |\beta_j|_\infty \leq M,\  i=1,\cdots,l,\ j=0,\cdots,l \}.
\end{align*} 
We then introduce the set $\Gc_{K,d,d',\ell,m}^{\zeta_\kappa}$ as
\begin{align}
     \Gc_{K,d,d',\ell,m}^{\zeta_\kappa} := & \{\Psi = (\Psi_i)_{i=1,\ldots,d'} : \R^d \mapsto \R^{d'},\ \Psi_i: x\in\R^d \mapsto K \beta_i  \;\;\;   \phi_i\Big(\frac{x+\alpha_i}{\beta_i}\Big)\in\R,\nonumber\\ & \quad \phi_i \in  \Sc_{d,\ell,m}^{\zeta_\kappa}, \ 
     \mbox{ for some } \alpha_i\in\R^d,\ \beta_i > 0\}.
\end{align}
Notice that these networks are $\sqrt{d'}K$-Lipschitz and that each of their components is $K$-Lipschitz. We rely on the the following quantitative approximation result which directly follows from \cite{TSB21}. 

\begin{Proposition}[Slight extension of Tanielian, Sangnier, Biau \cite{TSB21} : Approximation theorem for Lipschitz functions by Lipschitz GroupSort neural networks.]\label{prop: tsb}
    Let $f:[-R,R]^d\mapsto \R^{d'}$ be  $K$-Lipschitz. Then, for all $\varepsilon>0$, there exists a GroupSort neural network $g$ in $\Gc_{K,d,d',\ell,m}^{\zeta_\kappa}$ verifying \begin{equation}
    \sup_{x\in[-R,R]^d} |f(x)-g(x)|_2 \leq \sqrt{d'}2RK \varepsilon
\end{equation} with $g$ of grouping size $\kappa=\lceil \frac{2\sqrt{d}}{\varepsilon}\rceil$, depth $\ell + 1 = O(d^2)$ and width $\sum_{i=0}^{\ell-1} m_i = O((\frac{2\sqrt{d}}{\varepsilon})^{d^2-1})$ in the case $d>1$. If $d=1$, the same result holds with $g$ of grouping size $\kappa=\lceil \frac{1}{\varepsilon}\rceil$, depth $\ell + 1 = 3$ and width $\sum_{i=0}^{\ell-1} m_i =O(\frac{1}{\varepsilon}) $.
\end{Proposition}   
\begin{proof}
     With $f_i$ the $i$-th component of $f$, define \begin{align}\label{eq: tilde f}
        \widetilde{f}_i:z\in [0,1]^d \mapsto \frac{f_i(2R(z-1/2))}{2RK}.
    \end{align} Then $\widetilde{f}_i$ is 1-Lipschitz and by Theorem 3 from \cite{TSB21} if $d>1$ (or Proposition 5 from \cite{TSB21} if $d=1$), there exists a $1$-Lipschitz GroupSort neural network $g_i\in\Sc_{d,\ell,m}^{\zeta_\kappa}$ verifying \begin{equation}
    \sup_{z\in[0,1]^d} |\widetilde{f}_i(z) - g_i(z)| \leq  \varepsilon
\end{equation} with $g_i$ of grouping size $\kappa = O(\frac{2\sqrt{d}}{\varepsilon})$, depth $\ell + 1 = O(d^2)$ and width $\sum_{i=0}^{\ell-1} m_i =O((\frac{2\sqrt{d}}{\varepsilon})^{d^2-1})$(respectively grouping size $\kappa=O( \frac{1}{\varepsilon})$, depth $\ell + 1 = 3$ and width $\sum_{i=0}^{\ell-1} m_i$ $=$ $O(\frac{1}{\varepsilon})$ if $d=1$).  Inverting \eqref{eq: tilde f} we have  $ f_i(x) = 2KR \widetilde{f}_i(\frac{x +R}{2R})$ hence
\begin{equation}
    \sup_{x\in[-R,R]^d} \Big|f_i(x) - 2 KR g_i\Big(\frac{x +R}{2R}\Big)\Big| \leq  2KR\varepsilon.
\end{equation}
The result is proven by concatenating the $d'$  $K$-Lipschitz GroupSort networks $x\mapsto 2 KR 
g_i(\frac{x +R}{2R})$, $i=1,\cdots,d'$.
\end{proof}

\begin{Remark}
    As mentioned in \cite{TSB21}, GroupSort neural networks generalize the ReLU networks and, thanks to their Lipschitz continuity, offer better stability regarding noisy inputs and adversarial attacks. It also appears that GroupSort networks are more expressive than ReLU ones.
\end{Remark}

\subsection{Existing Schemes}

We review recent machine learning schemes that will serve as benchmarks for our new scheme described in the next section. All these schemes rely on  BSDE 
representation of the solution to the PDE, and differ according to the formulation of the time discretization of the  BSDE.  

For this purpose, let us introduce the diffusion process $\Xc$ in $\R^d$ associated to the linear part of the differential operator in  the PDE \eqref{eq: semilinear PDE}, namely: 
\begin{align} \label{eq:SDE}
\Xc_t &= \Xc_0 + \int_0^t \mu(s,\Xc_s) \di s +  \int_0^t \sigma(s,\Xc_s) \di W_s, \quad 0 \leq t \leq T, 
\end{align}
where $W$ is a $d$-dimensional standard Brownian motion on some  probability space  $(\Omega,\mathcal{F},\Pro)$ equipped with a filtration $\F$ $=$ $(\mathcal{F}_t)_t$, and 
$\Xc_0$ is an 
$\mathcal{F}_0$-measurable random variable valued in $\R^d$.   
Recall from  \cite{PP90} that  the solution $u$ to the PDE \eqref{eq: semilinear PDE} admits a proba\-bilistic representation in terms of the BSDE:
\begin{equation} \label{BSDEfeyn} 
Y_t = g(\Xc_T) - \int_t^T f(s,\Xc_s,Y_s,Z_s) \di s - \int_t^T Z_s.  \di W_s, \quad 0 \leq t \leq T, 
\end{equation} 
via the Feynman-Kac formula $Y_t$ $=$ $u(t,\Xc_t)$, $0\leq t \leq T$. When $u$ is a smooth function, this BSDE representation is directly obtained by It\^o's formula applied to 
$u(t,\Xc_t)$, and we have $Z_t$ $=$  $\sigma(t,\Xc_t)\trans D_x u(t,\Xc_t)$, $0\leq t \leq T$.

\vspace{1mm}

Let $\pi$ be a subdivision $\{t_0=0<t_1<\cdots<t_N=T\}$ with modulus  $|\pi| := \sup_i \Delta t_i$, $\Delta t_i$ $:=$ $t_{i+1} - t_i$, 
satisfying $|\pi| = O\left(\frac{1}{N}\right)$, and consider the Euler scheme
\begin{align} \label{eq: Euler}
X_{i} & = \Xc_0 +  \sum_{j=0}^{i-1}  \mu(t_j,X_{j}) \Delta t_j +  \sum_{j=0}^{i-1}    \sigma(t_j,X_{j}) \Delta W_j, \quad i=0,\ldots,N, 
\end{align} 
where  $\Delta W_j$ $:=$ $W_{t_{j+1}} - W_{t_j}$, $j$ $=$ $0,\ldots,N$.  When the diffusion $\Xc$ cannot be simulated, we shall rely on the 
simulated paths of $(X_i)_i$  that act as training data in the setting of machine learning, and thus our training set  can be chosen as large as desired.

The time discretization of the BSDE \eqref{BSDEfeyn} is written in backward induction as 
\begin{align} \label{Yinduc}
Y_{i}^\pi & = \; Y_{i+1}^\pi - f(t_i,X_i,Y_{i}^\pi,Z_{i}^\pi) \Delta t_i  - Z_{i}^\pi. \Delta W_{i}, \quad i=0,\ldots,N-1,  
\end{align}
which also reads as  conditional expectation formulae
\begin{equation} \label{Ycond}
\left\{
\begin{array}{ccl}
Y_{i}^\pi & = &  \E_i \Big[ Y_{{i+1}}^\pi - f(t_i,X_i,Y_{i}^\pi,Z_{i}^\pi) \Delta t_i  \Big]  \\
Z_{i}^\pi & = &  \E_i \Big[  \frac{\Delta W_i}{\Delta t_i} Y_{{i+1}}^\pi \Big], \quad \quad i=0,\ldots,N-1, 
\end{array}
\right.
\end{equation}
where $\E_i$ denotes the conditional expectation w.r.t. $\mathcal{F}_{t_i}$.  Alternatively, by iterating relations \eqref{Yinduc} together with the terminal relation $Y_{N}^\pi$ $=$ $g(X_N)$, we have 
\begin{align} \label{Yiter} 
Y_{i}^\pi & = \; g(X_N) - \sum_{j=i}^{N-1} \big[  f(t_j,X_j,Y_{j}^\pi,Z_{j}^\pi) \Delta t_j  +   Z_{j}^\pi. \Delta W_{j} \big], \quad i=0,\ldots,N-1. 
\end{align}

\vspace{2mm}

\noindent $\bullet$ {\bf Deep BSDE scheme  \cite{HJE17}.} 

\vspace{1mm}

\noindent The idea of the method is to treat the backward equation \eqref{Yinduc} as a forward equation by appro\-ximating the initial condition $Y_0$ and the  $Z$ component at each time  by networks functions of the 
$X$ process, so as to match the terminal condition.  More precisely, the problem is to minimize over  network functions $\Uc_0$ $:$ $\R^d$ $\rightarrow$ $\R$, and  sequences of network  functions 
$\bZc$ $=$ $(\Zc_i)_i$, $\Zc_i$ $:$ $\R^d$ $\rightarrow$  $\R^d$, $i$ $=$ $0,\ldots,N-1$,  the global quadratic loss function
\begin{align*}
J_G(\Uc_0,\bZc) &=\; \E \Big| Y_N^{\Uc_0,\bZc} - g(X_N) \Big|^2,
\end{align*}
where $(Y_i^{\Uc_0,\bZc})_i$ is defined by forward induction as 
\begin{align*}
Y_{i+1}^{\Uc_0,\bZc} & = \; Y_{i}^{\Uc_0,\bZc} +  f(t_i,X_i,Y_i^{\Uc_0,\bZc},\Zc_i(X_i)) \Delta t_i + \Zc_i(X_i).\Delta W_i, \quad i = 0,\ldots,N-1, 
\end{align*}
starting from $Y_0^{\Uc_0,\bZc}$ $=$ $\Uc_0(\Xc_0)$. 
The output of this scheme, for the solution $(\widehat{\Uc}_0,\widehat\bZc)$ to this global minimization problem, provides an approximation 
$\widehat{\Uc}_0$ of the solution $u(0,.)$ to the PDE at time $0$, and approximations $Y_i^{\widehat{\Uc}_0,\widehat\bZc}$  of the solution to the PDE \eqref{eq: semilinear PDE} at times $t_i$ evaluated at 
$\Xc_{t_i}$, i.e., of   $Y_{t_i}$ $=$ $u(t_i,\Xc_{t_i})$, $i$ $=$ $0,\ldots,N$.

\vspace{3mm}

\noindent $\bullet$  \textbf{Deep Backward Dynamic Programming (DBDP) \cite{HPW19}.}

\vspace{1mm}

\noindent  The method relies on  the backward dynamic programming relation \eqref{Yinduc}  arising from the time discretization of the BSDE, and learns simultaneously at each time step $t_i$ the 
pair $(Y_{t_i},Z_{t_i})$ with neural networks trained with the forward process $X$ and the Brownian motion $W$.  The  scheme has two versions: 
\begin{itemize}
	\item[1.] {\it  DBDP1}. Starting from $\widehat{\Uc}_N^{(1)}$ $=$ $g$, proceed by backward induction for $i$ $=$ $N-1,\ldots,0$, by minimizing over network functions 
	$\Uc_i$ $:$ $\R^d$ $\rightarrow$ $\R$, and $\Zc_i$ $:$ $\R^d$ $\rightarrow$ $\R^d$ the local quadratic  loss function 
	\begin{align*}
	J_i^{(B1)}(\Uc_i,\Zc_i) &=\; \E \Big| \widehat{\Uc}_{i+1}^{(1)}(X_{i+1}) -  \Uc_i(X_i) \\ & -  f(t_i,X_i,\Uc_i(X_i),\Zc_i(X_i)) \Delta t_i -  \Zc_i(X_i).\Delta W_i \Big|^2, 
	\end{align*}
	and update $(\widehat{\Uc}_i^{(1)},\widehat{\Zc}_i^{(1)})$ as the solution to this local minimization problem. 
	\item[2.] {\it  DBDP2}. Starting from $\widehat{\Uc}_N^{(2)}$ $=$ $g$, proceed by backward induction for $i$ $=$ $N-1,\ldots,0$, by minimizing over $C^1$ network functions  
	$\Uc_i$ $:$ $\R^d$ $\rightarrow$ $\R$ the local quadratic  loss function 
	\begin{align*}
	&J_i^{(B2)}(\Uc_i)  \\ &=\; \E \Big| \widehat{\Uc}_{i+1}^{(2)}(X_{i+1}) -  \Uc_i(X_i) -  f(t_i,X_i,\Uc_i(X_i), \sigma(t_i,X_i)\trans D_x\Uc_i(X_i)) \Delta t_i \\
	&  \quad \quad \quad - \;   D_x\Uc_i(X_i)\trans\sigma(t_i,X_i)\Delta W_i \Big|^2, 
	\end{align*}
	where $D_x \Uc_i$ is the automatic differentiation of the network function $\Uc_i$. Update $\widehat{\Uc}_i^{(2)}$ as  the solution to this problem, and set 
	$\widehat{\Zc}_i^{(2)}$ $=$ $\sigma\trans(t_i,.)D_x{\Uc}_i^{(2)}$. 
\end{itemize}
The output of DBDP  provides an approximation $(\widehat{\Uc}_i^{},\widehat{\Zc}_i^{})$  
of the solution $u(t_i,.)$ and its gradient $\sigma\trans(t_i,.)D_xu(t_i,.)$  to the PDE \eqref{eq: semilinear PDE} at times $t_i$, $i$ $=$ $0,\ldots,N-1$. 
The approximation error has been analyzed in \cite{HPW19}.

\begin{Remark}
	{\rm A machine learning scheme in the spirit of regression-based Monte-Carlo methods (\cite{BT04}, \cite{GLW05}) 
		for approximating condition expectations in the time discretization  \eqref{Ycond} of  the BSDE,  can be formulated as follows: 
		starting from $\hat\Uc_N$ $=$ $g$, proceed by backward induction for $i$ $=$ $N-1,\ldots,0$, in two regression  problems:  
		\begin{itemize}
			\item[(a)] Minimize over network functions $\Zc_i$ $:$ $\R^d$ $\rightarrow$ $\R^d$ 
			\begin{align*}
			J_{i}^{r,Z}(\Zc_i) & = \; \E \Big| \frac{\Delta W_i}{\Delta t_i} \widehat{\Uc}_{i+1}(X_{i+1}) -  \Zc_i(X_i) \Big|^2 
			\end{align*}
			and update $\widehat{\Zc}_i$ as  the solution to this minimization problem
			\item[(b)]  Minimize over network functions $\Uc_i$ $:$ $\R^d$ $\rightarrow$ $\R$
			\begin{align*}
			J_{i}^{r,Y}(\Uc_i) &= \; \E \Big|  \widehat{\Uc}_{i+1}^{}(X_{i+1}) -  \Uc_i(X_i) -  f(t_i,X_i,\Uc_i(X_i),\widehat{\Zc}_i(X_i)) \Delta t_i   \Big|^2
			\end{align*} 
			and update $\widehat{\Uc}_i$ as  the solution to this minimization problem. 
		\end{itemize}
		Compared to these regression-based schemes, the DBDP scheme approximates simultaneously the pair component $(Y,Z)$ via the minimization of the loss functions $J_i^{(B1)}(\Uc_i,\Zc_i)$ (or 
		$J_i^{(B2)}(\Uc_i)$ for the second version), $i$ $=$ $N-1,\ldots,0$. One advantage of this latter approach is that the accuracy of the DBDP scheme can be tested when computing at each time step the infimum of loss function, which should be equal to zero for the exact solution (up to the time discretization).   In contrast, the infimum of the  loss functions in the regression-based schemes is not  known for the exact solution as it corresponds in theory to the residual of $L^2$-projection, and thus the accuracy of the scheme cannot be tested directly in-sample.  Moreover, a variant where the automatic differentiation  $D_x\Uc_i(X_i) $ is performed to estimate $Z_{t_i}$ instead of using a second neural network $\widehat{\Zc}_i$ (similarly as in the previous DBDP2 scheme) can also be considered. In this case, one only needs to solve for each time step the (b) optimization problem and not the (a) problem anymore.
	}
	\ep
\end{Remark}

\vspace{2mm}

\noindent $\bullet$ \textbf{Deep Splitting (DS) scheme \cite{BBCJN19}.} 

\vspace{1mm}

\noindent This method also proceeds by backward induction as follows:  
\begin{itemize}
	\item[-] Minimize  over  $C^1$ network functions $\Uc_N$ $:$ $\R^d$ $\rightarrow$ $\R$ the terminal loss function
	\begin{align*}
	J_N^S(\Uc_N) &= \;  \E\Big| g(X_N) - \Uc_N(X_N) \Big|^2,
	\end{align*}
	and denote by $\widehat{\Uc}_N$ as  the solution to this minimization problem. If $g$ is $C^1$, we can choose directly $\widehat{\Uc}_N$ $=$ $g$.
	\item[-] For $i$ $=$ $N-1,\ldots,0$, minimize over $C^1$ network functions $\Uc_i$ $:$ $\R^d$ $\rightarrow$ $\R$ the loss function 
	\begin{align} 
	& J_i^S(\Uc_i) \nonumber \\
	& = \;  \E \Big|  \widehat{\Uc}_{i+1}^{}(X_{i+1}) -  \Uc_i(X_i) \nonumber  \\ & \quad \quad \quad -  f(t_i,X_{i+1},\widehat{\Uc}_{i+1}(X_{i+1}),\sigma(t_i,X_i)\trans D_x\widehat{\Uc}_{i+1}(X_{i+1})) \Delta t_i   \Big|^2, \label{lossDS} 
	\end{align}

	and update $\widehat\Uc_i$ as  the solution to this minimization problem.  Here $D_x$ refers again to the automatic differentiation operator for network functions.  

\end{itemize}
The DS scheme combines ideas  of the DBDP2 and  regression-based schemes where the current regression-approximation on $Z$ is replaced by the automatic differentiation of the 
network function computed at the previous step. The current approximation of $Y$ is then computed by a  regression network-based scheme.  
In Section \ref{sec: convergence}, we shall analyze  the approximation error of the DS scheme. Please note that in \eqref{lossDS} we consider a slight modification of the original DS scheme from \cite{BBCJN19}. In their loss function, the term $f(t_i,X_{i+1},\widehat{\Uc}_{i+1}(X_{i+1}),\sigma(t_i,X_i)\trans D_x\widehat{\Uc}_{i+1}(X_{i+1}))$ is replaced by  $f(t_{i+1},X_{i+1},\widehat{\Uc}_{i+1}(X_{i+1}),\sigma(t_{i+1},X_{i+1})\trans D_x\widehat{\Uc}_{i+1}(X_{i+1}))$.

\subsection{Deep Backward Multi-step Scheme (MDBDP)} \label{subsec:MDBDP}

The starting point of the MDBDP  scheme is the iterated representation \eqref{Yiter} for the time discretization of the BSDE. 
%combines features of the DBDP  scheme  presented in the previous section and multi-step methods. 
This backward scheme is described as follows: for $i$ $=$ $N-1,\ldots,0$, minimize 
over network functions $\Uc_i$ $:$ $\R^d$ $\rightarrow$ $\R$, and $\Zc_i$ $:$ $\R^d$ $\rightarrow$ $\R^d$ the loss function 
\begin{align} 
J_i^{MB}(\Uc_i,\Zc_i) & = \; \E \Big| g(X_N) - \sum_{j=i+1}^{N-1} f(t_j,X_j,\widehat{\Uc}_j(X_j),\widehat{\Zc}_j(X_j))\Delta t_j - \sum_{j=i+1}^{N-1} \widehat{\Zc}_j(X_j). \Delta W_j  \nonumber \\
& \quad \quad \quad - \;  \Uc_i(X_i) -  f(t_i,X_i,\Uc_i(X_i),\Zc_i(X_i)) \Delta t_i -  \Zc_i(X_i).\Delta W_i \Big|^2  \label{JMB}
\end{align}
and update $(\widehat{\Uc}_i^{},\widehat{\Zc}_i^{})$ as the solution to this  minimization problem.  
This output  provides an approxi\-mation $(\widehat{\Uc}_i,\widehat{\Zc}_i)$ of the solution $u(t_i,.)$ to the PDE \eqref{eq: semilinear PDE} at times $t_i$, $i$ $=$ $0,\ldots,N-1$. This  
appro\-ximation error  will be analyzed in Section \ref{sec: convergence}. 

MDBDP  is a machine learning version of the Multi-step Dynamic Programming method studied by \cite{BD07} and \cite{GT14}.  
Instead of solving at each time step two regression problems, our approach allows to consider only a single minimization as in the DBDP scheme. Compared to the latter, the multi-step consideration is expected to provide better accuracy by reducing the propagation of errors in the backward induction.

	\begin{Remark}\label{rem: DBDP2}
		We could have also considered, as in the DBDP2 scheme, the automatic differentiation of $\widehat{\Uc}_i$ for the approximation of the gradient $Z_{t_i}$. However, as shown in the numerical tests of \cite{HPW19}, this approach leads to less accurate results %\blue{worse !! not as good.. } 
		than the DBDP1 algorithm which uses an additional neural network.  Moreover, at least for theoretical analysis, it requires to optimize over $C^1$ neural networks, which is a restrictive assumption. Hence we focus on a DBDP1-type method.
\end{Remark}

In the numerical implementation, the expectation defining the loss function $J_i^{MB}$ in \eqref{JMB} is replaced by an empirical average leading to the so-called {\it generalization} (or estimation) error, 
largely studied in the statistical community, see \cite{gyo02}, and more recently \cite{huretal18}, \cite{becjenkuc20} and the references therein.    
Moreover,  recalling the parametrization $(\Uc^\theta,\Zc^\theta)$ of neural network functions in $\Nc_{d,1,\ell,m}^\rho\times\Nc_{d,d,\ell,m}^\rho$, the minimization of the empirical average is amenable to stochastic gradient descent (SGD) extensively used in machine learning. More precisely, given a fixed time step $i$ $=$ $N-1,\ldots,0$, at each iteration  of the SGD,  
we pick a sample $(X_j^k,\Delta W_j^k)_{j=i,\ldots,N}$ of the Euler process and increment of Brownian motion $(X_j,\Delta W_j)_j$, $k$ $=$ $1,\ldots,K$, of mini-batch size $K$,  and 
consider the empirical loss function: 
\begin{align}
&\J_i^{K}(\theta) \nonumber \\ & = \;  \frac{1}{K} \sum_{k=1}^K 
\Big| g(X_N^k) - \sum_{j=i+1}^{N-1} f(t_j,X_j^k,\widehat{\Uc}_j(X_j^k),\widehat{\Zc}_j(X_j^k))\Delta t_j - \sum_{j=i+1}^{N-1} \widehat{\Zc}_j(X_j^k). \Delta W_j^k  \nonumber \\
& \hspace{3cm}  - \;  \Uc^\theta(X_i^k) -  f(t_i,X_i^k,\Uc^\theta(X_i^k),\Zc^\theta(X_i^k)) \Delta t_i -  \Zc^\theta(X_i^k).\Delta W_i^k  \Big|^2, \label{Jtheta} 
\end{align}
where $\widehat{\Uc}_j$ $=$ $\Uc_j^{\hat\theta_j}$, $\widehat{\Zc}_j$ $=$ $\Zc_j^{\hat\theta_j}$, and $\hat\theta_j$ is the resulting parameter from the SGD obtained at dates $j$ $\in$ $\llbracket i+1,N-1\rrbracket$. 
%$=$ $i+1,\ldots,N-1$. 
In practice, the number of iterations for SGD 
at the initial induction time $N-1$ should be large enough so as to learn accurately the  value function $u(t_{N-1},.)$ and its gradient $D_x u(t_{N-1},.)$ via 
$\widehat{\Uc}^{\hat\theta_{N-1}}$ and $\widehat{\Zc}^{\hat\theta_{N-1}}$. However,  it is then expected 
%(by continuity in time of the value function) 
that $(\widehat{\Uc}_j,\widehat{\Zc}_j)$ does not vary 
a lot  from $j$ $=$ $i+1$ to $i$, which means that at time $i$, one can design the SGD with initialization parameter equal to the resulting parameter from  the previous SGD at  time $i+1$, and then use few iterations to obtain accurate values of $\widehat{\Uc_i}$ and $\widehat{\Zc}_i$.  This observation allows to reduce significantly the computational time in (M)DBDP scheme when applying sequentially  
$N$ SGD.  The SGD  algorithm for computing an approximate minimizer of the loss function induces the so-called {\it optimization} error, which has been extensively studied in the  stochastic algorithm and machine learning communities, see \cite{bachmou13}, \cite{berfor13},  \cite{becjenkuc20}, and the references therein.

\vspace{3mm}

\begin{algorithm2e} \label{AlgoMDBDP} 
	\DontPrintSemicolon 
	\SetAlgoLined 
	\vspace{1mm}
	{\bf Data:} 
	Initial parameter $\hat\theta_N$. A sequence of number of iterations $(S_i)_{i=0,\ldots,N-1}$ 
	
	\vspace{0.5mm}
	
	\For{ $i$ $=$ $N-1,\ldots,0$} 
	{ Initial parameter $\theta_i$ $\leftarrow$ $\hat\theta_{i+1}$   
		
		Set $s$ $=$ $1$
		
		\While{ $s$ $\leq$ $S_i$} 
		{Pick a  sample of $(X_j,\Delta W_j)_{j=i,\ldots,N}$ of mini-batch size $K$
			
			Compute the gradient $\nabla \J_i^{K}(\theta)$ of $\J_i^{K}(\theta)$ defined in \eqref{Jtheta}
			
			Update $\theta_i$ $\leftarrow$ $\theta_i - \eta \nabla \J_i^{K}(\theta_i)$ \mbox{ with } $\eta$ learning rate 
			
			$s$ $\leftarrow$ $s+1$ 
			
		}
		Return $\hat\theta_i$ $\leftarrow$ $\theta_i$, $\widehat{\Uc}_i$ $=$ $\Uc^{\hat\theta_i}$,  $\widehat{\Zc}_i$ $=$ $\Zc^{\hat\theta_i}$ \tcc*{Update parameter, function and derivative} 
	}
	\caption{MDBDP  scheme. \label{MDBDPscheme}}
\end{algorithm2e}

\section{Convergence Analysis}\label{sec: convergence}

This section is devoted to the approximation error  and rate of convergence  of the MDBDP, DS, and DBDP schemes 
described in Section \ref{sec:semischeme}.

\vspace{1mm}

We make the following standard assumptions on the coefficients of the  forward-backward equation associated to semilinear PDE \eqref{eq: semilinear PDE}.  

\vspace{1mm}

\begin{Assumption} \label{H1}
	\begin{itemize}
		\item[(i)] $\Xc_0$ is square-integrable : $\Xc_0$ $\in$ $L^2(\Fc_0,\R^d)$.
		\item[(ii)] The functions $\mu$ and $\sigma$  are Lipschitz in $x$ $\in$ $\R^d$, uniformly in $t$ $\in$ $[0,T]$. 
		\item[(iii)] The generator function $f$ is $1/2$-H\"older continuous in time and Lipschitz continuous in all other variables: $\exists$ $[f]_{_L}$ $>$  $0$ such that for all 
		$(t,x,y,z)$ and $(t',x',y',z')$ $\in$ $[0,T]\times\R^d\times\R\times\R^d$,
		\begin{align*}
		& |f(t,x,y,z) - f(t',x',y',z')|  \\ &\leq \;  [f]_{_L} \big( |t-t'|^{1/2}+ |x-x'|_{_2} +|y-y'|+ |z-z'|_{_2} \big).
		\end{align*}
		Moreover,  $\sup_{t\in[0,T]} |f(t,0,0,0)|$ $<$  $\infty$. 
		\item[(iv)]  The function $g$ satisfies a linear growth condition. 
	\end{itemize}
\end{Assumption}

Assumption  \ref{H1}  guarantees the  existence and uniqueness of an adapted solution $(\Xc,Y,Z)$ to the forward-backward equation \eqref{eq:SDE}-\eqref{BSDEfeyn}, satisfying 
\begin{align*}
\E \Big[ \sup_{0\leq t\leq T}|\Xc_t |_{_2}^2 + \sup_{0\leq t\leq T}|Y_t|^2 + \int_0^T |Z_t|^2_{_2} \di t \Big] & < \; \infty, 
\end{align*}
( see for instance Theorem 3.3.1, Theorem 4.2.1, Theorem 4.3.1  from \cite{zhangBook}). 
Given the time grid $\pi$ $=$ $\{t_i: i=0,\ldots,N\}$, let us introduce the $L^2$-regularity of $Z$:
\begin{align*}
\varepsilon^Z(\pi) & := \; \E \bigg[ \sum_{i=0}^{N-1} \int_{t_i}^{t_{i+1}} |Z_t - \bar Z_{t_i} |_{_2}^2 \di t \bigg],
\;\;\; \mbox{ with } \; \bar Z_{t_i} \; := \; \frac{1}{\Delta t_i} \E_i \Big[  \int_{t_i}^{t_{i+1}} Z_t \di t \Big]. 
\end{align*}
Since $\bar Z$ is a  $L^2$-projection of $Z$, we know that $\eps^Z(\pi)$ converges to zero when $|\pi|$ goes to zero. Moreover, as shown in \cite{Z04},  
when $g$ is also Lipschitz, we have 
\begin{align*}
\eps^Z(\pi)  & = \; O(|\pi|). 
\end{align*}
Here, the standard notation $O(|\pi|)$ means that  $\limsup_{|\pi| \rightarrow 0}  \; |\pi|^{-1} O(|\pi|)$ $<$ $\infty$.

\begin{Lemma}\label{lem: euler Lipschitz}
Under Assumption \ref{H2} (ii), the following standard estimate for the Euler-Maruyama scheme holds when $\Delta t_i\rightarrow 0$ 
\begin{equation}
    \E|X_{i+1}^{x} - X_{i+1}^{x'} |_2^2 \leq (1+C\Delta t_i) |x-x'|_2^2,
\end{equation} where $X_{i+1}^{x} := x + \mu(t_i,x)\Delta t_i +  \sigma(t_i,x)\Delta W_{i}$.
\end{Lemma}\begin{proof}
    By expanding the square, simply notice that the dominant terms when $\Delta t_i\rightarrow 0$ are of older $\Delta t_i$ because the term of order $\sqrt{\Delta t_i}$, namely $ (x-x')\cdot(\sigma(t_i,x)-\sigma(t_i,x'))\Delta W_{i}$ has a null expectation and all other terms are dominated by $\Delta t_i$.
\end{proof}

\subsection{Convergence of the MDBDP Scheme}\label{sec: convergence multi step}

We fix classes of  functions $\Nc_i$ and $\Nc'_i$ 
for the approximations respectively of the solution and its gradient, and define  $(\widehat{\Uc}_i^{(1)},\widehat{\Zc}_i^{(1)})$ as the  output of the MDBDP scheme at times $t_i$, $i$ $=$ $0,\ldots,N$. 

\vspace{1mm}

Let us define (implicitly)  the  process 
\begin{equation}  \label{defVi}
\left\{
\begin{array}{ccl}
V_{i}^{(1)} & = & \E_i\Big[g(X_N) - f\big(t_{i},X_i,V_{i}^{(1)},\overline{\widehat{Z}_{i}}^{(1)}\big)\Delta t_i - \Sum_{j=i+1}^{N-1} f\big(t_{j},X_j,\widehat{\Uc}_{j}^{(1)}(X_{j}),\widehat{\Zc}_{j}^{(1)}(X_{j})\big)\Delta t_j \Big],  \\
\overline{\widehat{Z}}_i^{(1)} & = &  \E_i\Big[ \frac{g(X_N)\Delta W_i}{\Delta t_i} - \Sum_{j=i+1}^{N-1} f\big(t_{j},X_j,\widehat{\Uc}_{j}^{(1)}(X_{j}),\widehat{\Zc}_{j}^{(1)}(X_{j})\big)
\frac{\Delta W_i \Delta t_j}{\Delta t_i}  \Big],  \ i=0,\ldots,N,  
\end{array}
\right.
\end{equation}
and notice  by the Markov property of the discretized forward process $(X_i)_i$ that 
\begin{align}\label{eq: Markov functions}
V_{i}^{(1)}  \; = \;  v_i^{(1)}(X_i),  & \quad \quad  \overline{\widehat{Z}_{i}}^{(1)} \; = \;  \widehat{z_i}^{(1)}(X_i), \quad i=0,\ldots,N, 
\end{align} 
for some  deterministic functions $v_i^{(1)}, \widehat{z_i}^{(1)}$.  Let us then  introduce 
\begin{align*}
\varepsilon_i^{1,y} \; := \;  \inf_{ \Uc_{}\in\Nc_i} \E \big|v_i^{(1)}(X_i) - \Uc_{}(X_i) \big|^2,  & \quad \quad  
\varepsilon_i^{1,z} \; := \;  \inf_{\Zc_{}\in\Nc_i'} \E \big| \widehat{z_i}^{(1)}(X_i) - \Zc_{}(X_i) \big|_{_2}^2,
%\end{cases}
\end{align*}
for $i$ $=$ $0,\ldots,N-1$, which represent the $L^2$-approximation errors of the functions $v_i^{(1)}, \widehat{z_i}^{(1)}$ in the classes  $\Nc_i$ and $\Nc'_i$.

\begin{Theorem}[Approximation error of MDBDP] \label{theoMDBDP} 
	Under Assumption \ref{H1},  there exists a constant $C> 0$ (depending only on the data $\mu,\sigma,f,g,d,T$) such that in the limit $ |\pi| \rightarrow 0$  
	\begin{align} 
	&  \sup_{i\in\llbracket 0,N \rrbracket}\ \E\big|Y_{t_i} - \widehat{\Uc}_{i}^{(1)}(X_{i})\big|^2 +  \E\Big[\sum_{i=0}^{N-1}\int_{t_i}^{t_{i+1}} \big|Z_s - \widehat{\Zc}^{(1)}_i(X_i)\big|_{_2}^2\ \di s\Big] \nonumber  \\
	& \leq \;  C \Big( \E\big|g(\Xc_T) - g(X_N) \big|^2   + |\pi|+ \varepsilon^Z(\pi) + \sum_{j=0}^{N-1} (\varepsilon_j^{1,y} + \Delta t_j \varepsilon_j^{1,z}) \Big). \label{estimMDBDP}
	\end{align}
\end{Theorem}

\begin{Remark}
	{\rm 
		The upper bound in \eqref{estimMDBDP} 
		consists of four terms. 
		The first three terms correspond to the time discretization of BSDE, similarly as in  \cite{Z04}, \cite{BT04},  
		namely (i) the strong approximation of the terminal condition (depending on the forward scheme and  $g$), and converging to zero, as $|\pi|$ goes to zero,  
		with a rate $|\pi|$ when $g$ is Lipschitz, (ii) the strong approximation of the forward Euler scheme, and the $L^2$-regularity of $Y$, which gives a convergence of order $|\pi|$, 
		(iii) the $L^2$-regularity of $Z$.  
		Finally, the last term is  the approximation error by  the chosen class of functions. 
		Note that the approximation error  $\sum_{j=0}^{N-1} (\varepsilon_j^{1,y} + \Delta t_j \varepsilon_j^{1,z})$ in \eqref{estimMDBDP} is better than the one for the  DBDP scheme derived in  \cite{HPW19}, 
		with an order $\sum_{j=0}^{N-1} (N\varepsilon_j^{1,y} + \varepsilon_j^{1,z})$.  In the work \cite{GT14} which introduced the multistep scheme with linear regression, the authors noticed the same improvement in the error propagation in comparison with the one-step classical scheme \cite{GLW05}.
	}
	\ep
\end{Remark}

\vspace{1mm}

We next study convergence for the approximation error of the MDBDP  scheme, for a specific choice of functions classes $\Nc_i$ and $\Nc'_i$ and with the additional assumption that $f$ does not depend on $z$.
\begin{Assumption} \label{H2}
		The generator function $f$ is independent of $z$. Namely, for all 
		$(t,x,y,z,z')$ $\in$ $[0,T]\times\R^d\times\R\times\R^d\times\R^d$,
		\begin{align*}
	f(t,x,y,z) = f(t,x,y,z').
		\end{align*}
\end{Assumption}

Actually, if $f$ is linear in $z$: $f(t,x,y,z)$ $=$ $\bar f(t,x,y) + \lambda(t,x).z$, one can boil down to Assumption \ref{H2} for $\bar f$ by incorporating the linearity in the drift function $\mu$, namely with the modified drift:  $\bar \mu(t,x) = \mu(t,x) - \sigma\lambda(t,x)$.

\begin{Proposition}[Rate of convergence of MDBDP]\label{cor: convergence MDBDP} 
	Let  Assumption \ref{H1} and Assumption \ref{H2} hold, and assume that $\Xc_0$ $\in$ $L^{2+\delta}(\Fc_0,\R^d)$,  for some $\delta$ $>$ $0$, and   $g$ is $[g]-$Lipschitz. 
	Then, there exists a bounded sequence $K_i$ (uniformly in $i,N$) such that for GroupSort neural networks classes $\Nc_i$ $=$ $\Gc_{K_i,d,1,\ell,m}^{\zeta_\kappa}$, and  $\Nc_i'$ $=$ $\Gc_{\sqrt{\frac{d}{\Delta t_i}}K_i,d,d,\ell,m}^{\zeta_\kappa}$, we have 
	\begin{align*}
	\sup_{i\in\llbracket0,N\rrbracket} \E\big|Y_{t_{i}}-\widehat{\Uc}^{(1)}_{i}(X_{i})\big|^2 +  \E\Big[\sum_{i=0}^{N-1}\int_{t_i}^{t_{i+1}} \big|Z_s - \widehat{\Zc}_i^{(1)}(X_i)\big|_{_2}^2\ \di s\Big]  
	& = \;  O(1/N), 
	\end{align*}
	with a grouping size $\kappa=O( 2\sqrt{d} N^{2})$, depth $\ell + 1 = O(d^2)$ and width $\sum_{i=0}^{\ell-1} m_i = O((2\sqrt{d} N^{2})^{d^2-1})$ in the case $d>1$. If $d=1$,  take $\kappa=O( N^{2})$, depth $\ell + 1 = 3$ and width $\sum_{i=0}^{\ell-1} m_i =O(N^{2}) $.
	Here, the constants in the $O(\cdot)$ term  depend only on  $\mu,\sigma,f,g$, $d,T,\Xc_0$.
\end{Proposition}

\subsection{Convergence of the DS Scheme}

We consider classes $\Nc_i^{\gamma,\eta}$ of differentiable $\gamma_i-$Lipschitz functions with $\eta_i-$Lipschitz derivative for sequences $\gamma$ $=$ $(\gamma_i)_i$, $\eta$ $=$ $(\eta_i)_i$ and define  $\widehat{\Uc}_i^{(2)}$ as the  output of the DS scheme at times $t_i$, $i$ $=$ $0,\ldots,N$. . 

\vspace{1mm}

Let us define the process 
\begin{align}  \label{defViDS} 
V_i^{(2)} & = \; \E_i\Big[\widehat{ \Uc}_{i+1}^{(2)}(X_{i+1}) - f\big(t_{i},X_i,\E_i[\widehat{\Uc}_{i+1}^{(2)}(X_{i+1})],\E_i[\sigma(t_i,X_i)\trans D_x[\widehat{\Uc}_{i+1}^{(2)}(X_{i+1})]\big)\Delta t_i \Big], 
\end{align} 
for $i$ $\in$ $\llbracket 0,N-1\rrbracket$, and $V_N^{(2)}$ $=$ $\widehat{\Uc}_N^{(2)}(X_N)$.  By the Markov property of $(X_i)_i$, we have $V_i^{(2)}$ $=$ $v_i^{(2)}(X_i)$, for some  functions 
$v_i^{(2)}$ $:$ $\R^d$ $\rightarrow$ $\R$, $i$ $\in$ $\llbracket 0,N-1\rrbracket$, and 
we introduce 
\begin{align*}
& \varepsilon_i^{\gamma,\eta} = \;  
\left\{
\begin{array}{ll} 
\inf_{\Uc \in \Nc_i^{\gamma,\eta}} \E\big| v_i^{(2)}(X_i) - \Uc(X_i) \big|^2, &  i=0,\ldots,N-1, \\
\inf_{\Uc \in \Nc_i^{\gamma,\eta}} \E\big| g(X_N) - \Uc(X_N) \big|^2, & i=N.
\end{array}%\; = \;   \E\big| g(X_N) - \widehat{\Uc}_N(X_N) \big|^2
\right. 
\end{align*}  
the $L^2$-approximation error in the class $\Nc_i^{\gamma,\eta}$ of the functions $v_i^{(2)}$, $i$ $=$ $0,\ldots,N-1$, and $g$.

\begin{Theorem}[Approximation error of DS]\label{theo: CV DS} 
	Let Assumption \ref{H1} hold, and assume that $\Xc_0$ $\in$ $L^4(\Fc_0,\R^d)$. Then, 
	there exists a constant $C> 0$ (depending only on $\mu,\sigma,f,g,d,T,\Xc_0$) such that in the limit $  |\pi| \rightarrow 0$  
	\begin{align}
	\sup_{i\in\llbracket0,N\rrbracket} \E\big|Y_{t_{i}}-\widehat{\Uc}_{i}^{(2)}(X_{i})\big|^2 
	& \leq C \Big( \E\big|g(X_N)-g(\Xc_T)\big|^2 + |\pi| +  \varepsilon^Z(\pi)  \nonumber  \\
	&  \quad \quad  \quad + \;   \max_i\big[\gamma_i^2, \eta_i^2 \big] |\pi| +   \varepsilon_N^{\gamma,\eta} +  N  \sum_{i=0}^{N-1} \varepsilon_i^{\gamma,\eta}  \Big).  \label{errorDS} 
	\end{align} 
\end{Theorem}

\vspace{1mm}

\begin{Remark}
	{\rm
		We retrieve a similar error as in the analysis of the  DBDP2 scheme  derived in \cite{HPW19}. Notice that when $g$ is $C^1$, one can choose to initialize  the DS scheme with $\widehat{\Uc}_N$ $=$ $g$, and the term $\varepsilon_N^{\gamma,\eta}$ is removed in \eqref{errorDS}. 
	}
	\ep
\end{Remark}

The GroupSort neural networks being only continuous but not differentiable, we are not able to express a convergence rate for the Deep Splitting scheme in terms of the architecture and number of neurons to choose, like in Propositions \ref{cor: convergence MDBDP}, \ref{cor: convergence DBDP}. It would require a quantitative approximation result for $C^1$ neural networks with bounded Lipschitz gradient, and this is left to future research.

\subsection{ Convergence of the DBDP Scheme}\label{sec: convergence one step}

We consider classes of  functions $\Nc_i$ and $\Nc_i'$
for the approximations of the solution and its gradient, and define  $(\widehat{\Uc}_i^{(3)},\widehat{\Zc}_i^{(3)})$ as the  output of the DBDP scheme at times $t_i$, $i$ $=$ $0,\ldots,N$. 
\vspace{1mm}
Let us define (implicitly)  the  process 
\begin{equation} 
\left\{
\begin{array}{ccl}
V_{i}^{(3)}  & = & \E_i\Big[\widehat{\Uc}_{i+1}^{(3)}(X_{i+1} )-  f\big(t_{i},X_i ,V_{i}^{(3)} ,\overline{\widehat{Z}_{i}}^{(3)} \big)\Delta t_i \Big]\\
\overline{\widehat{Z}_{i}}^{(3)}  & = &  \E_i\Big[\widehat{\Uc}_{i+1}^{(3)}(X_{i+1} ) \frac{\Delta W_i}{\Delta t_i}\Big], \quad i = k,\ldots,N-1.
\end{array}
\right.
\end{equation}
and notice  by the Markov property of the discretized forward process $(X_i)_i$ that 
\begin{align}\label{eq : Markov functions 3}
V_{i}^{(3)}  \; = \;  v_i^{(3)}(X_i),  & \quad \quad  \overline{\widehat{Z}_{i}} \; = \;  \widehat{z_i}^{(3)}(X_i), \quad i=0,\ldots,N, 
\end{align} 
for some  deterministic functions $v_i^{(3)}, \widehat{z_i}^{(3)}$.  Let us then  introduce 
\begin{align*}
\varepsilon_i^{3,y} \; := \;  \inf_{ \Uc_{}\in\Nc_i} \E \big|v_i^{(3)}(X_i) - \Uc_{}(X_i) \big|^2,  & \quad \quad  
\varepsilon_i^{3,z} \; := \;  \inf_{\Zc_{}\in\Nc_i'} \E \big| \widehat{z_i}^{(3)}(X_i) - \Zc_{}(X_i) \big|_{_2}^2,
\end{align*}
for $i$ $=$ $0,\ldots,N-1$, which represent the $L^2$-approximation errors of the functions $v_i^{(3)}, \widehat{z_i}^{(3)}$ in the classes  $\Nc_i$ and $\Nc_i'$.  
\begin{Theorem}[Huré, Pham, Warin \cite{HPW19} : Approximation error of DBDP] \label{theoDBDP} 
	Under Assumption \ref{H1},  there exists a constant $C> 0$ (depending only on the data $\mu,\sigma,f,g,d,T$) such that in the limit $ |\pi| \rightarrow 0$  
	\begin{align} 
	&  \sup_{i\in\llbracket 0,N \rrbracket}\ \E\big|Y_{t_i} - \widehat{\Uc}_{i}^{(3)}(X_{i})\big|^2 +  \E\Big[\sum_{i=0}^{N-1}\int_{t_i}^{t_{i+1}} \big|Z_s - \widehat{\Zc}_i^{(3)}(X_i)\big|_{_2}^2\ \di s\Big] \nonumber  \\
	& \leq \;  C \Big( \E\big|g(\Xc_T) - g(X_N) \big|^2   + |\pi|+ \varepsilon^Z(\pi) + N\sum_{j=0}^{N-1} (\varepsilon_j^{3,y} + \Delta t_j \varepsilon_j^{3,z}) \Big). \label{estimDBDP}
	\end{align}
\end{Theorem}

\vspace{2mm}

We next study convergence rate for the approximation error of the DBDP scheme, and need to specify the class of network functions $\Nc_i$ and $\Nc_i'$. 
\begin{Proposition}[Rate of convergence of DBDP]\label{cor: convergence DBDP} 
	Let  Assumption \ref{H1} hold, and assume that $\Xc_0$ $\in$ $L^{2+\delta}(\Fc_0,\R^d)$,  for some $\delta$ $>$ $0$, and   $g$ is $[g]-$Lipschitz. Then, there exists a bounded sequence $K_i$ (uniformly in $i,N$) such that for $\Nc_i$ $=$ $\Gc_{K_i,d,1,\ell,m}^{\zeta_\kappa}$, and  $\Nc_i'$ $=$ $\Gc_{\sqrt{\frac{d}{\Delta t_i}}K_i,d,d,\ell,m}^{\zeta_\kappa}$, we have 
	\begin{align*}
	\sup_{i\in\llbracket0,N\rrbracket} \E\big|Y_{t_{i}}-\widehat{\Uc}_{i}^{(3)}(X_{i})\big|^2 +  \E\Big[\sum_{i=0}^{N-1}\int_{t_i}^{t_{i+1}} \big|Z_s - \widehat{\Zc}_i^{(3)}(X_i)\big|_{_2}^2\ \di s\Big]  
	& = \;  O(1/N), 
	\end{align*} with a grouping size $\kappa = O( 2\sqrt{d} N^{3})$, depth $\ell + 1 = O(d^2)$ and width $\sum_{i=0}^{\ell-1} m_i = O((2\sqrt{d} N^{3})^{d^2-1})$ in the case $d>1$. If $d=1$,  take $\kappa=O( N^{3})$, depth $\ell + 1 = 3$ and width $\sum_{i=0}^{\ell-1} m_i =O(N^{3}) $.
	Here, the constants in the $O(\cdot)$ term  depend only on  $\mu,\sigma,f,g$, $d,T,\Xc_0$.  
\end{Proposition}

\section{Proof of the Main Theoretical  Results} \label{sec:proof}

\subsection{Proof of Theorem \ref{theoMDBDP}}

Let us introduce the processes $(\bar V_i,\bar Z_i)_i$ arising from the time discretization of the BSDE \eqref{BSDEfeyn}, and defined by the {\it implicit} backward Euler scheme:  
\begin{equation} \label{defbarVi} 
\left\{
\begin{array}{ccl}
\bar V_{i}^{(1)} &= &  \E_i\Big[ \bar{V}_{i+1}^{(1)} - f\big(t_{i},X_i,\bar V_{i}^{(1)},\bar Z_i^{(1)}\big)\Delta t_i\Big]\\
\bar Z_{i}^{(1)} &= & \E_i\Big[\bar{V}_{i+1}^{(1)} \frac{\Delta W_i}{\Delta t_i} \Big], \quad i=0,\ldots,N-1,  
\end{array} 
\right.
\end{equation}
starting from $\bar V_N^{(1)}$ $=$ $g(X_N)$. We recall from \cite{Z04} the time discretization error:  
\begin{align}
& \sup_{i \in \llbracket0,N \rrbracket}   \E\big|Y_{t_i} -  \bar{V_{i}}^{(1)}\big|^2 + 
\E\Big[\sum_{i=0}^{N-1} \int_{t_i}^{t_{i+1}} \big|Z_s - \bar Z_i^{(1)}\big|_{_2}^2\ \di s\Big] \nonumber 
\\ & \leq   C \Big( \E\big|g(\Xc_T) - g(X_N) \big|^2 + |\pi|  + \varepsilon^Z(\pi) \Big),  \label{eq: time discretization error}
\end{align} 
for some constant $C$  depending only on the coefficients satisfying Assumption \ref{H1}.

Let us introduce the auxiliary  process
\begin{align} \label{barViter}
\hat{V}_{i}^{(1)} &= \;  \E_i\Big[g(X_N) - \Sum_{j=i}^{N-1} f\big(t_{j},X_j,\widehat{\Uc}_{j}^{(1)}(X_{j}),\widehat{\Zc}_{j}^{(1)}(X_{j})\big)\Delta t_j \Big], \quad i=0,\ldots,N, 
\end{align}
and notice by the tower property of conditional expectations that we have the recursive relations:
\begin{align}\label{eq: one step representation}
\hat{V_{i}}^{(1)} &= \; \E_i\Big[\hat{V}_{i+1}^{(1)} -  f\big(t_{i},X_i,\widehat{\Uc}_{i}^{(1)}(X_{i}),\widehat{\Zc}_{i}^{(1)}(X_{i})\big) \Delta t_i \Big], \quad  i =0,\ldots,N-1. 
\end{align} 
Observe also that  $\overline{\widehat{Z}}_i^{(1)}$ defined in \eqref{defVi} satisfies
\begin{align} \label{eq: one step representationZ}
\overline{\widehat{Z}_{i}}^{(1)} &= \; \E_i\Big[\hat{V}_{i+1}^{(1)} \frac{\Delta W_i}{\Delta t_i}\Big], \quad i =0,\ldots,N-1.
\end{align}

We now decompose the approximation error, for  $i\in\llbracket0,N-1\rrbracket$,  into  
\begin{align}
& \E\big|Y_{t_i} - \widehat{\Uc}_{i}^{(1)}(X_{i})\big|^2 \\ & \leq \;  
4\Big(\E\big|Y_{t_i} -  \bar{V_{i}}^{(1)} \big|^2 + \E\big| \bar{V_{i}}^{(1)} - \hat{V}_{i}^{(1)}\big|^2  + \E\big|\hat{V}_{i}^{(1)} - V_{i}^{(1)}\big|^2 + \E\big|V_{i}^{(1)} - \hat{\Uc}_{i}^{(1)}(X_{i})\big|^2\Big) \nonumber  \\
& = : \; 4(I_i^1 + I_i^2 + I_i^3 + I_i^4), \label{eq: error decomposition}
\end{align} 
and analyze each of these contribution terms. In the sequel,  $C$ denotes a generic constant independent of $\pi$ that may vary from line to line,  and depending only on the coefficients satisfying Assumption \ref{H1}.  
Notice that the first contribution term is  the time discretization error for BSDE given by \eqref{eq: time discretization error}, and we shall study the three other terms in the following steps.   

\vspace{1mm} 

\noindent \textit{\underline{Step 1.}}  Fix $i$ $\in$ $\llbracket0,N-1\rrbracket$. 
From the definition \eqref{defVi} of $V_i^{(1)}$ and by the martingale representation theorem,  
there exists a square integrable process $\{\widehat{Z}_s^{(1)},t_i\leq s \leq T\}$ s.t. 
\begin{align}
& g(X_N) - f\big(t_{i},X_i,V_{i}^{(1)},\overline{\widehat{Z}_{i}^{(1)}}\big)\Delta t_i 
- \sum_{j=i+1}^{N-1} f\big(t_{j},X_j,\widehat{\Uc}_{j}^{(1)}(X_{j}),\widehat{\Zc}_{j}^{(1)}(X_{j})\big)\Delta t_j  \nonumber \\ & = \;  V_i + \int_{t_i}^{t_N} \widehat{Z}_s^{(1)}.  \di W_s.  \label{decmartin} 
\end{align}
From the definition \eqref{defVi} of $\overline{\widehat{Z}_{i}}^{(1)}$, and  by  It\^o isometry, we then have 
\begin{align} \label{decZ} 
\overline{\widehat{Z}_{i}}^{(1)} &= \;  \frac{ \E_i\big[\int_{t_i}^{t_{i+1}} \widehat{Z}_s^{(1)}  \di s \big]}{ \Delta t_i},  
\;  \mbox{ i.e. }  \; \E_i\big[  \int_{t_i}^{t_{i+1}} \big( \widehat{Z}_s^{(1)} -  \overline{\widehat{Z}_{i}}^{(1)} \big)  \di s \big] \; = \;  0.    
\end{align}
Plugging \eqref{decmartin}  into  \eqref{JMB}, we see that the loss function of the MDBDP scheme can be rewritten as 
\begin{align}
& \Jc_i^{MB}(\Uc_i,\Zc_i) \\ &=  \;  \E\Big|V_i^{(1)} - \Uc_{i}(X_{i})   + \Delta t_i \big[ f\big(t_{i},X_i,V_i^{(1)},\overline{\widehat{Z}_{i}}^{(1)}\big)
-  f\big(t_{i},X_i,\Uc_{i}(X_{i}),\Zc_{i}(X_{i})\big)  \big]  \nonumber   \\ 
& \quad \quad  + \;   \sum_{j=i+1}^{N-1} \int_{t_j}^{t_{j+1}} \big[ \widehat{Z}_s^{(1)} - \widehat{\Zc}_{j}(X_{j})\big]. \di W_s  + 
\int_{t_i}^{t_{i+1}} \big[ \widehat{Z}_s^{(1)} - \Zc_{i}(X_{i}) \big] . \di W_s \Big|^2 \nonumber \\
& =   \widetilde{\Jc}_i^{MB}(\Uc_i,\Zc_i) +  \E \Big[ \sum_{j=i}^{N-1} \int_{t_j}^{t_{j+1}} \big| \widehat{Z}_s^{(1)} - \overline{\widehat{Z}_j}^{(1)} \big|_{_2}^2\di s \Big] \nonumber 
\\ &  \quad +  \sum_{j=i+1}^{N-1} \Delta t_j \E \big| \overline{\widehat{Z}_{j}}^{(1)} - \widehat{\Zc}_{j}(X_{j}) \big|_{2}^2,  \label{JtildeJ} 
\end{align} 
where we use \eqref{decZ}, and 
\begin{align*}
& \widetilde{\Jc}_i^{MB}(\Uc_i,\Zc_i)\\  &:=  \E\Big|V_i^{(1)} - \Uc_{i}(X_{i}) +\Delta t_i\big[ f\big(t_{i},X_i,V_i^{(1)},\overline{\widehat{Z}_{i}}^{(1)}\big) - f\big(t_{i},X_i,\Uc_{i}(X_{i}),\Zc_{i}(X_{i})\big) \big] \Big|^2 \\
& \quad  \quad + \;  \Delta t_i \E\big| \overline{\widehat{Z}_{i}}^{(1)} - \Zc_{i}(X_{i})\big|_{_2}^2. 
\end{align*}
It is clear by Lipschitz continuity of $f$ in Assumption \ref{H1} that 
\begin{align} \label{interJtilde}
\widetilde{\Jc}_i^{MB}(\Uc_i,\Zc_i)  & \leq \;  C\Big( \E\big|V_i^{(1)} - \Uc_{i}(X_{i})\big|^2 + \Delta t_i \E\big| \overline{\widehat{Z}_{i}}^{(1)} - \Zc_{i}(X_{i})\big|_{_2}^2\Big).
\end{align}
On the other hand, by the Young inequality:  $(1 - \beta) a^2$ $+$ $\big(1 - \frac{1}{\beta}\big) b^2$ $\leq$  
$(a+b)^2$ $\leq$ $(1+\beta)a^2$ $+$ $\big(1 + \frac{1}{\beta} \big) b^2$, for all $(a,b)$ $\in$ $\R^2$, and  $\beta >0$, we have 
\begin{align}
& \widetilde{\Jc}_i^{MB}(\Uc_i,\Zc_i)\\  &\geq \; (1-\beta) \E\big|V_i^{(1)} - \Uc_{i}(X_{i})\big|^2 + \;  \Delta t_i \E\big| \overline{\widehat{Z}_{i}}^{(1)} - \Zc_{i}(X_{i})\big|_{_2}^2 \nonumber \\ 
&+ \quad \quad \Big(1- \frac{1}{\beta}\Big) |\Delta t_i|^2 
\E\big|f\big(t_{i},X_i,\Uc_{i}(X_{i}),\Zc_{i}(X_{i})\big) - f\big(t_{i},X_i,V_i^{(1)},\overline{\widehat{Z}_{i}}^{(1)}\big) \big|^2    \nonumber \\ 
& \geq  \; (1-\beta) \E\big|V_i^{(1)} - \Uc_{i}(X_{i})\big|^2 + \;  \Delta t_i\E\big| \overline{\widehat{Z}_{i}}^{(1)} - \Zc_{i}(X_{i})\big|_{_2}^2  \nonumber   \\ 
& \quad \quad -  \frac{2[f]_{_L}^2}{\beta}  |\Delta t_i|^2 \Big(\E\big|\Uc_{i}(X_{i})-V_i^{(1)}\big|^2 
+ \E\big|\Zc_{i}(X_{i})-\overline{\widehat{Z}_{i}}^{(1)} \big|_{_2}^2\Big)  \nonumber \\          
& \geq \;  \Big(1-\big(4 [f]_{_L}^2+\frac{1}{2} \big)\Delta t_i\Big) \E\big|V_i^{(1)} - \Uc_{i}(X_{i})\big|^2 + \frac{1}{2}\Delta t_i\E\big| \overline{\widehat{Z}_{i}}^{(1)} - \Zc_{i}(X_{i})\big|_{_2}^2, \label{tildeJlower} 
\end{align}
where we use the Lipschitz continuity of $f$ in the second inequality, and choose explicitly $\beta = 4 [f]_{_L}^2 \Delta t_i $ ($<$ $1$ for $\Delta t_i $ small enough) in the last one.  
By applying inequality \eqref{tildeJlower} to $(\Uc_i,\Zc_i)$ $=$ $(\widehat{\Uc}_i^{(1)},\widehat{\Zc}_i^{(1)})$, which  is a minimizer of $\tilde{\Jc}_i^{MB}$ by \eqref{JtildeJ}, and combining with  
\eqref{interJtilde}, this yields  for $\Delta t_i$ small enough and for all functions $\Uc_{i}$, $\Zc_{i}$: 
\begin{align*}
& \E\big|V_i^{(1)} - \widehat{\Uc}_{i}^{(1)}(X_{i})\big|^2 + \Delta t_i\E\big| \overline{\widehat{Z}_{i}}^{(1)} - \widehat{\Zc}_{i}^{(1)}(X_{i})\big|_{_2}^2 \\ & \leq \;  
C\Big(\E\big| V_i - \Uc_{i}(X_{i}) \big|^2 + \Delta t_i \E\big| \overline{\widehat{Z}_{i}} - \Zc_{i}(X_{i})\big|_{_2}^2\Big).    
\end{align*}    
By minimizing over $\Uc_{i},\Zc_{i}$ in the right hand side, we get the approximation error in the classes $\Nc_i,\Nc_i'$ of the regressed functions $V_i^{(1)}$, $\overline{\widehat{Z}_{i}}^{(1)}$:
\begin{align}
\E\big|V_i^{(1)} - \widehat{\Uc}_{i}^{(1)}(X_{i})\big|^2 + \Delta t_i\E\big| \overline{\widehat{Z}_{i}}^{(1)} - \widehat{\Zc}_{i}^{(1)}(X_{i})\big|_{_2}^2    
& \leq \;  C(\varepsilon^{1,y}_i + \Delta t_i \varepsilon^{1,z}_i).  \label{eq: regression error}
\end{align}

\vspace{1mm}

\noindent \textit{\underline{Step 2.}} From the expressions of $V_i^{(1)}$ and $\hat{V}_i^{(1)}$ in \eqref{defVi}, \eqref{barViter}, and by Lipschitz continuity of $f$, we have by \eqref{eq: regression error}: 
\begin{align}
\E\big|\hat{V_{i}}^{(1)} - V_{i}^{(1)}\big|^2 & = \;  \Delta t_i^2 \E\Big|\E_i\big[f\big(t_{i},X_i,V_i^{(1)},\overline{\widehat{Z}_{i}}^{(1)}\big)
- f\big(t_{i},X_i,\widehat{\Uc}_{i}^{(1)}(X_{i}),\widehat{\Zc}_{i}^{(1)}(X_{i})\big)  \big]\Big|^2 \nonumber \\
&\leq \;  2 [f]_{_L}^2 |\Delta t_i|^2 \Big(\E\big|V_i^{(1)} - \widehat{\Uc}_{i}^{(1)}(X_{i})\big|^2 +  \E\big| \overline{\widehat{Z}_{i}}^{(1)} - \widehat{\Zc}_{i}^{(1)}(X_{i})\big|_{_2}^2 \Big) \nonumber \\
&\leq \;  C \Delta t_i (\varepsilon_i^{1,y} + \Delta t_i \varepsilon_i^{1,z}), \quad \quad  i = 0,\ldots,N.  \label{eq: auxiliary error}
\end{align}

%\vspace{1mm}

\noindent \textit{\underline{Step 3.}}  
From the  recursive expressions of $\bar{V_{i}}^{(1)}$, $\hat{V}_i^{(1)}$ in \eqref{defbarVi}, \eqref{eq: one step representation},   
and applying the Young, the Cauchy-Schwarz inequalities,   together with the  Lipschitz condition of $f$, we get  for $\beta>0$:

\begin{align}
& \E\big|\bar V_{i}^{(1)} - \hat V_{i}^{(1)}\big|^2 \nonumber  \\ &\leq \;  
(1+\beta ) \E\Big|\E_i\big[\bar V_{i+1}^{(1)} - \hat V_{i+1}^{(1)}\big]\Big|^2 + 2 [f]_{_L}^2 \Big(1 + \frac{1}{\beta}\Big)|\Delta t_i|^2
\Big( \E\big|\bar{V_{i}}^{(1)} - \hat{\Uc}_{i}^{(1)}(X_{i})\big|^2 +  \E\big|\bar Z_i^{(1)} - \widehat{\Zc}_{i}^{(1)}(X_{i})\big|_{_2}^2 \Big) \nonumber \\
&\leq \; (1+\beta) \E\Big|\E_i\big[\bar V_{i+1}^{(1)} - \hat V_{i+1}^{(1)} \big]\Big|^2 + 2 [f]_{_L}^2 \Big(1 + \frac{1}{\beta}\Big)|\Delta t_i|^2
\big(  3 \E|\bar{V_{i}}^{(1)} - \hat{V_{i}}^{(1)}|^2 +  2 \E\big|\bar Z_i^{(1)} - \overline{\widehat{Z}_{i}}^{(1)}\big|_{_2}^2  \big) \nonumber \\
&  \; + \;   2 [f]_{_L}^2 \Big(1 + \frac{1}{\beta}\Big)|\Delta t_i|^2 \Big( 3 \E|\hat{V_{i}}^{(1)} - V_{i}^{(1)}|^2 +  3 \E|V_{i}^{(1)} - \widehat{\Uc}^{(1)}_{i}(X_{i})|^2 
+ 2 \E\big|\overline{\widehat{Z}_{i}}^{(1)} - \widehat{\Zc}_{i}(X_{i})\big|_{_2}^2  \Big) \nonumber \\
& \leq \;   (1+\beta) \E\Big|\E_i\big[\bar V_{i+1}^{(1)} - \hat V_{i+1}^{(1)} \big]\Big|^2 +  (1 + \beta)\frac{2[f]_{_L}^2 |\Delta t_i|^2}{\beta}
\big(  3 \E|\bar{V_{i}}^{(1)} - \hat{V_{i}}^{(1)}|^2 +  2   \E\big|\bar Z_i^{(1)} - \overline{\widehat{Z}_{i}}^{(1)}\big|_{_2}^2  \big) \nonumber \\
& \; + \; C[f]_{_L}^2 \Big(1 + \frac{1}{\beta}\Big) \Delta t_i (\varepsilon_i^{1,y} + \Delta t_i \varepsilon_i^{1,z}), \label{barVinter} 
\end{align}
where we use \eqref{eq: regression error}, \eqref{eq: auxiliary error} in the last inequality.  Moreover, by \eqref{defbarVi}, \eqref{eq: one step representationZ}, we have 
\begin{align*}
\Delta t_i \big( \bar Z_i^{(1)} - \overline{\widehat{Z}_{i}}^{(1)} \big)  & = \; \E_i\Big[ \Delta W_i \big( \bar V_{i+1}^{(1)} - \hat V_{i+1}^{(1)} \big) \Big] \\
&= \; \E_i\Big[ \Delta W_i \Big( \bar V_{i+1}^{(1)} - \hat V_{i+1}^{(1)}  - \E_i\big[ \bar V_{i+1}^{(1)} - \hat V_{i+1}^{(1)}  \big]\Big) \Big], 
\end{align*}
and thus by the Cauchy-Schwarz inequality 
\begin{align}
\Delta t_i \E\big| \bar Z_i^{(1)} - \overline{\widehat{Z}_{i}}^{(1)} \big|^2_{_2} & \leq \;  d \Big( \E\big|\bar V_{i+1}^{(1)} - \hat V_{i+1}^{(1)} \big|^2 - \E\Big| \E_i\big[ \bar V_{i+1}^{(1)} - \hat V_{i+1}^{(1)}  \big]\Big|^2 \Big).  \label{eq: Z propagation}  
\end{align} 
Plugging into \eqref{barVinter}, and choosing $\beta$ $=$ $4d[f]_{_L}^2 \Delta t_i$, gives
\begin{align*}
& (1- C \Delta t_i)  \E\big|\bar V_{i}^{(1)} - \hat V_{i}^{(1)} \big|^2 \\ &\leq \;    (1 + C \Delta t_i)  \E\big|\bar V_{i+1}^{(1)} - \hat V_{i+1}^{(1)}\big|^2 + 
(1 + C \Delta t_i)  \big(\varepsilon_i^{1,y} + \Delta t_i \varepsilon_i^{1,z})
\end{align*}
By discrete Gronwall lemma, and recalling that $\bar V_N^{(1)}$ $=$ $\hat{V}_N^{(1)}$ ($=$  $g(X_N)$), we then obtain
\begin{align} \label{estimV4}
\sup_{i \in \llbracket0,N \rrbracket}   \E\big|\bar{V_{i}}^{(1)} - \hat{V_{i}}^{(1)}\big|^2 &\leq \; C \sum_{i=0}^{N-1}  \big(\varepsilon_i^{1,y} + \Delta t_i \varepsilon_i^{1,z}). 
\end{align} 
The required bound for the approximation error  on $Y$ follows by plugging \eqref{eq: time discretization error}, \eqref{eq: regression error}, \eqref{eq: auxiliary error}, and \eqref{estimV4} into 
\eqref{eq: error decomposition}.   

\vspace{1mm}

\noindent \textit{\underline{Step 4.}}  
We decompose the approximation error for the $Z$ component into three terms
\begin{align}
& \E\Big[\sum_{i=0}^{N-1}\int_{t_i}^{t_{i+1}} \big|Z_s^{(1)} - \widehat{\Zc}_i^{(1)}(X_i)\big|_{_2}^2\ \di s\Big]  \nonumber \\
&\leq \;  3 \sum_{i=0}^{N-1}\Big( \E\Big[\int_{t_i}^{t_{i+1}} \big|Z_s^{(1)} - \bar Z_i^{(1)}\big|_{_2}^2\ \di s\Big] 
+ \Delta t_i \E\big| \bar Z_i^{(1)}- \overline{\widehat{Z}}_i^{(1)}\big|_{_2}^2 + \Delta t_i \E\big| \overline{\widehat{Z}}_i^{(1)}-\widehat{\Zc}_i^{(1)}(X_i)\big|_{_2}^2 \Big).   \label{eq: decomposition Z error}
\end{align}  
By summing the inequality \eqref{eq: Z propagation} (recalling that $\bar V_N^{(1)}$ $=$ $\hat{V}_N^{(1)}$), and using \eqref{barVinter}, we have for $\beta\in(0,1)$:

\begin{align}
& \sum_{i=0}^{N-1} \Delta t_i \E| \bar Z_i^{(1)}- \overline{\widehat{Z}}_i^{(1)}|^2 \nonumber \\
& \leq \;  d\sum_{i=0}^{N-1} \Big( \E\big|\bar V_{i}^{(1)} - \hat V_{i}^{(1)} \big|^2 - \E\big| \E_i\big[ \bar V_{i+1}^{(1)} - \hat V_{i+1}^{(1)}  \big]\big|^2 \Big) \nonumber \\
& \leq \; d\sum_{i=0}^{N-1} \Big(\beta \E\Big|\E_i\big[\bar V_{i+1}^{(1)} - \hat V_{i+1}^{(1)} \big]\Big|^2 +  \big(1 + \frac{1}{\beta}\big)\big(2[f]_{_L}^2 |\Delta t_i|^2\big)
\big(  3 \E|\bar{V_{i}}^{(1)} - \hat{V_{i}}^{(1)}|^2 +  2   \E\big|\bar Z_i^{(1)} - \overline{\widehat{Z}_{i}}^{(1)}\big|_{_2}^2  \big) \nonumber \\
& \quad \quad \quad    + \; C[f]_{_L}^2 \big(1 + \frac{1}{\beta}\big) \Delta t_i (\varepsilon_i^{1,y} + \Delta t_i \varepsilon_i^{1,z}) \Big) \nonumber  \\
& \leq \;  d\sum_{i=0}^{N-1} \Big(\frac{8d[f]_{_L}^2 \Delta t_i}{1-8d[f]_{_L}^2 \Delta t_i} \E\big|\E_i\big[\bar V_{i+1}^{(1)} - \hat V_{i+1}^{(1)} \big]\big|^2 
+ \frac{3}{4d}\Delta t_i \E|\bar{V_{i}}^{(1)} - \hat{V_{i}}^{(1)}|^2 \; + \;  \frac{C}{8d } 
(\varepsilon_i^{1,y} + \Delta t_i \varepsilon_i^{1,z})\Big) \nonumber \\
& \quad \quad \quad + \;  \frac{1}{2} \sum_{i=0}^{N-1} \Delta t_i \E\big| \bar Z_i^{(1)}- \overline{\widehat{Z}}_i^{(1)}\big|_{_2}^2,  \label{eq: propagation Z approximation}
\end{align} 
by choosing explicitly $\beta = \frac{8d[f]_{_L}^2 \Delta t_i}{1-8d[f]_{_L}^2 \Delta t_i} = O(\Delta t_i)$ for $\Delta t_i$ small enough. Plugging \eqref{eq: time discretization error}, \eqref{eq: regression error}, \eqref{estimV4}, and \eqref{eq: propagation Z approximation}  (using the Jensen inequality) into \eqref{eq: decomposition Z error},  this proves the required bound for the approximation error on $Z$, and 
completes the proof.
\ep

\subsection{Proof of Proposition \ref{cor: convergence MDBDP}}
 Let us introduce the flow of the Euler scheme $(X_i)$ by:
\begin{align*}
X_{j+1}^{k,x} &:=  X_{j}^{k,x}  + \mu(t_j,X_{j}^{k,x}) \Delta t_j + \sigma(t_j,X_{j}^{k,x}) \Delta W_j,\ j=k,\ldots,N,
\end{align*} 
starting from $X_k^{k,x}$ $=$ $x$ at time step $j$ $=$ $k$ $\in$ $\N^*$. 
Under Assumption \ref{H2}, $f$ does not depend on $z$ so by slight abuse of notation we write $f(t,x,y) = f(t,x,y,z)$. Define
\begin{equation*}  
\left\{
\begin{array}{ccl}
V_{i,1}^{k,x} & = & \E_i\Big[g(X_{N}^{k,x}) - f\big(t_{i},X_{i}^{k,x},V_{i,1}^{k,x}\big)\Delta t_i - \Sum_{j=i+1}^{N-1} f\big(t_{j},X_{j}^{k,x},\widehat{\Uc}_{j}^{(1)}(X_{j}^{k,x})\big)\Delta t_j \Big],  \\
\hat{V}_{i,1}^{k,x} & = &  \E_i\Big[g(X_{N}^{k,x}) - \Sum_{j=i}^{N-1} f\big(t_{j},X_{j}^{k,x},\widehat{\Uc}_{j}^{(1)}(X_{j}^{k,x})\big)\Delta t_j \Big], \quad i=k,\ldots,N, \\
\overline{\widehat{Z}}_{i,1}^{k,x} & = &  \E_i\Big[ V_{i+1,1}^{k,x}
\frac{\Delta W_i}{\Delta t_i}  \Big],\ i=k,\ldots,N,  
\end{array}
\right.
\end{equation*}
and observe that we have the recursive relations:
\begin{align*}
\hat{V}_{i,1}^{k,x} &= \; \E_i\Big[\hat{V}_{i+1,1}^{k,x} -  f\big(t_{i},X_{i}^{k,x},\widehat{\Uc}_{i}^{(1)}(X_{i}^{k,x})\big) \Delta t_i \Big],  \ i=k,\ldots,N, \\
V_{i,1}^{k,x} &= \; \E_i\Big[\hat{V}_{i+1,1}^{k,x} -  f\big(t_{i},X_{i}^{k,x},V_{i,1}^{k,x}\big)\Delta t_i  \Big],  \ i=k,\ldots,N.
\end{align*} Notice by the Markov property of the discretized forward process $(X_i^{k,x})_i$ that 
\begin{align*}
V_{j,1}^{k,x}  \; = \;   v_{j}^{(1)}(X_{j}^{k,x}),\ \hat{V}_{j,1}^{k,x}  \; = \;  \hat{v}_{j}^{(1)}(X_{j}^{k,x}),\ \overline{\widehat{Z}}_{j,1}^{k,x}  \; = \;  \hat z_{j}^{(1)}(X_{j}^{k,x}),\ j=k,\ldots,N
\end{align*} 
for some deterministic function $v_{j}^{(1)}, \hat v_{j}^{(1)}, \hat z_{j}^{(1)}$ which do not depend on $k$. Notably $v_{j}^{(1)},\hat z_{j}^{(1)}$ are the same functions as in \eqref{eq: Markov functions}.

\noindent \textit{\underline{Step 1.}} We first estimate the evolution of the Lipschitz constant of $\hat v_i^{(1)}$  when $i$ varies. Let $x'\in\R^d$. By the Cauchy-Schwarz inequality
\begin{align}
   \Delta t_k \E\Big|\overline{\widehat{Z}}_{k,1}^{k,x} - \overline{\widehat{Z}}_{k,1}^{k,x'} \Big|^2  \leq \frac{1}{\Delta t_k} \E\Big|\E_k\Big[(\hat{V}_{k+1,1}^{k,x} - \hat{V}_{k+1,1}^{k,x'})  \Delta W_k \Big]\Big|^2
    \leq d\ \E\Big|\hat{V}_{k+1,1}^{k,x} - \hat{V}_{k+1,1}^{k,x'}\Big|^2 .\label{eq: Z Lips 1}
\end{align} Moreover, assuming that $\widehat{\Uc}_{k}^{(1)}$ is $[\widehat{\Uc}_{k}^{(1)}]-$Lipschitz yields
\begin{align*}
     \E\Big|\hat{V}_{k,1}^{k,x} - \hat{V}_{k,1}^{k,x'} \Big| & \leq  \E\Big|\hat{V}_{k+1,1}^{k,x}-\hat{V}_{k+1,1}^{k,x'}\Big| + \Delta t_i\E\Big| \{f\big(t_{k},x',\widehat{\Uc}_{k}^{(1)}(x')\big)-f\big(t_{k},x,\widehat{\Uc}_{k}^{(1)}(x)\big)\}  \Big|\\
    & \leq  \E\Big|\hat{V}_{k+1,1}^{k,x}-\hat{V}_{k+1,1}^{k,x'}\Big|+ [f]\Delta t_i (1+[\widehat{\Uc}_{k}^{(1)}]) |x-x'|_2.\label{eq: prop lipschitz hat}
\end{align*}
\noindent \textit{\underline{Step 2.}} 
Then for the $v_{k}^{(1)}$ function, the Young inequality gives
\begin{align*}
    \E\Big|V_{k,1}^{k,x} - V_{k,1}^{k,x'} \Big|^2 & \leq (1+\gamma\Delta t_k) \E\Big|\E_k\Big[\hat{V}_{k+1,1}^{k,x}-\hat{V}_{k+1,1}^{k,x'}\Big]\Big|^2 \\ & + (1+\frac{1}{\gamma\Delta t_k})\Delta t_k^2\E\Big|\{f\big(t_{k},x',V_{k,1}^{k,x'}\big) -f\big(t_{k},x,V_{k,1}^{k,x}\big)  \}  \Big|^2\\%%%%%%%%%%%%%%
    & \leq (1+\gamma\Delta t_k) \E\Big|\E_k\Big[\hat{V}_{k+1,1}^{k,x}-\hat{V}_{k+1,1}^{k,x'}\Big]\Big|^2 \\ & + 2[f]^2(1+\frac{1}{\gamma\Delta t_k})\Delta t_k^2\E\big[|x-x'|_2^2 + |V_{k,1}^{k,x}-V_{k,1}^{k,x'}|^2  \big].  
\end{align*}Therefore by choosing $\gamma = 2[f]^2$ for $\Delta t_k$ small enough 
\begin{align*}
    & \E\Big|V_{k,1}^{k,x} - V_{k,1}^{k,x'} \Big|^2  \leq (1+(\gamma+3)\Delta t_k) \E\Big|\hat{V}_{k+1,1}^{k,x}-\hat{V}_{k+1,1}^{k,x'}\Big|^2  + (1+(\gamma+3)\Delta t_i)\Delta t_{k}|x-x'|_2^2.
\end{align*}
Hence assuming $\hat v_{k+1}^{(1)}$ is $[\hat v_{k+1}^{(1)}]-$Lipchitz we obtain with Lemma \ref{lem: euler Lipschitz}
\begin{align}
  |v_{k}^{(1)}(x) - v_{k}^{(1)}(x')|^2 & = \E\Big|V_{k,1}^{k,x} - V_{k,1}^{k,x'} \Big|^2 \\& \leq (1+(\gamma+3)\Delta t_k)((1+C\Delta t_k)[\hat v_{k+1}^{(1)}]^2 + \Delta t_{k}) |x-x'|_2^2 \nonumber\\ & \leq (1+\tilde C\Delta t_k)([\hat v_{k+1}^{(1)}]^2 + \Delta t_{k}) |x-x'|_2^2 := [v_{k}^{(1)}]^2 |x-x'|_2^2,\label{eq: Lipschitz propagation}
\end{align} for $\Delta t_k$ small enough and another constant $\tilde C$. \\
\noindent \textit{\underline{Step 3.}}  Let $\epsilon>0,\ \kappa\in\N,\ \ell\in\N,\ m\in\R^\ell$ to be chosen after. Recursively, we choose $\Nc_k$ $=$ $\Gc_{[v_k^{(1)}],d,1,\ell,m}^{\zeta_{\kappa}}$ (with $[v_{N-1}^{(1)}]^2 = (1+\tilde C\Delta t_{N-1})([g]^2+\Delta t_{N-1})$ by \eqref{eq: Lipschitz propagation}) to approximate $v_k^{(1)}$ by $[v_k^{(1)}]-$Lipschitz GroupSort neural networks with uniform error $2[v_k]R\epsilon $ on $[-R,R]^d$ , see Proposition \ref{prop: tsb}. 
Therefore, by Lemma \ref{lem: euler Lipschitz}, estimations \eqref{eq: prop lipschitz hat} and the definition of $[v_k^{(1)}]$ in \eqref{eq: Lipschitz propagation}, for $\Delta t_k$ small enough
\begin{align*}
  |\hat v_{k}^{(1)}(x) - \hat v_{k}^{(1)}(x')|  & \leq \E\Big|\hat{V}_{k+1,1}^{k,x}-\hat{V}_{k+1,1}^{k,x'}\Big| + [f]\Delta t_k (1+[\widehat{\Uc}_{k}^{(1)}]) |x-x'|_2 \\
  & \leq (1+(C+2[f])\Delta t_k) [\hat v_{k+1}^{(1)}]|x-x'|_2 + [f](1+C\Delta t_k)\Delta t_k  |x-x'|_2.
\end{align*}  Thus $\hat v_k^{(1)}$ is $[\hat v_k^{(1)}]$ Lipschitz with 
\begin{align*}
  [\hat v_{k}^{(1)}] & \leq (1+\hat C\Delta t_k) [\hat v_{k+1}^{(1)}] + [f](1+C\Delta t_k)\Delta t_i
\end{align*} for a constant $\hat C$.
 By discrete Gronwall lemma over $k=N-1,\ldots,0$, 
\begin{align*}
  [\hat v_i^{(1)}]^2  \leq K,\
  [v_i^{(1)}]^2  \leq K,
\end{align*} uniformly in $i, N$ for some constant $K$. By \eqref{eq: Z Lips 1} and Proposition \ref{prop: tsb}, we choose $\Nc_k'$ $=$ $\Gc_{\sqrt{\frac{d}{\Delta t_i}}[v_k^{(1)}],d,d,\ell,m}^{\zeta_{\kappa}}$ to approximate $ \widehat{z_{k}}^{(1)}$ by GroupSort neural networks with uniform error $2\frac{d}{\sqrt{\Delta t_k}}[v_k]R\epsilon $ on $[-R,R]^d$. Thus $\sqrt{\Delta t_k} \widehat{z_{k}}^{(1)},\sqrt{\Delta t_k} \Zc_k^{(1)}$ are $d K$ Lipschitz, uniformly. \\
\noindent \textit{\underline{Step 4.}} The regression errors $\varepsilon_i^{1,y}$ verify from, localization of $X_i$ on  $B_2(R)$,  
the H\"older inequality, and the Markov inequality, the approximation error of $v_i^{(1)}$,  $i\in\llbracket0,N-1\rrbracket$,  by the class of GroupSort neural networks (Proposition \ref{prop: tsb})
\begin{align}
\sqrt{ \varepsilon_i^{1,y} } & = \;  \inf_{\Uc_{}\in\Gc_{[v_i],d,1}} \big\| v_i^{(1)}(X_i) - \Uc_{}(X_i)\big\|_{_2}  \nonumber \\
& \leq \;  \inf_{\Uc_{}\in\Gc_{[v_i],d,1}}  \Big\| \big(v_i^{(1)}(X_i) - \Uc(X_i)\big) {\bf 1}_{_{X_i \in B_2(R)}} \Big\|_{_2}  \; +   
\Big\| \big(v_i^{(1)}(X_i) -  \widehat{\Uc}_{i}^{(1)}(X_i)\big) {\bf 1}_{_{|X_i|_{_2}\geq R}} \Big\|_{_2}   \nonumber \\
& \leq \;  2KR\epsilon  + \;   \E\Big|\big(v_i^{(1)}(X_i) -  \widehat{\Uc}^{(1)}_{i}(X_i)\big)^{2q}\Big|^{1/{2q}} \E\Big|{\bf 1}_{_{|X_i|_{_2}\geq R}}^\frac{2q}{2q-1} \Big|^{\frac{2q-1}{2q}} \nonumber\\
& = \;  2KR\epsilon  + \;\E\Big|\big(v_i^{(1)}(X_i) -  \widehat{\Uc}_{i}^{(1)}(X_i)\big)^{2q}\Big|^{1/{2q}} \E[{\bf 1}_{_{|X_i|_{_2}\geq R}}]^{\frac{2q-1}{2q}} \nonumber\\
& \leq \;  2KR\epsilon  + \;   \frac{ \big( \big\|v_i^{(1)}(X_i)-v_i^{(1)}(0)\|_{_{2q}}  + \big\|\widehat{\Uc}_{i}^{(1)}(X_i)-v_i^{(1)}(0)\big\|_{_{2q}} \big) \big\|X_i\big\|_{_\frac{2 q}{2q-1}} }{R},  \label{Chebichev} 
\end{align} 
for $q$ $>$ $1$ and $2q =  2+\delta$ with $\delta$ as in the statement of the Proposition and by noticing that $(v_i^{(1)}(X_i) -  \widehat{\Uc}^{(1)}_{i}(X_i)\big)=(v_i^{(1)}(X_i)-v_i^{(1)}(0) -  (\widehat{\Uc}^{(1)}_{i}(X_i)-v_i^{(1)}(0))\big)$.
Now, by Lipschitz continuity of $v_i^{(1)},\widehat{\Uc}^{(1)}$ and because $0\in B_2(R)$ 
we have 
\begin{align}
& \big\|\widehat{\Uc}^{(1)}_{i}(X_i)-v_i^{(1)}(0)\big\|_{_{2q}} + \big\|v_{i}^{(1)}(X_i)-v_i^{(1)}(0)\big\|_{_{2q}} \\ &\leq \big\|\widehat{\Uc}^{(1)}_{i}(0)-v_i^{(1)}(0)\big\|_{_{2q}} + \big\|\widehat{\Uc}^{(1)}_{i}(X_i)-\widehat{\Uc}_{i}^{(1)}(0)\big\|_{_{2q}} +  \big\|v_{i}^{(1)}(X_i)-v_i^{(1)}(0)\big\|_{_{2q}} \\ & \leq 2KR\epsilon  + 2 K \big\|X_i\big\|_{_{2q}}.\label{eq: lipschitz zero}
\end{align}  
Recalling the standard estimate $\|X_i\|_{_{2q}}$ $\leq$ $C(1+\|\Xc_0\|_{_{2q}})$, $i$ $=$ $0,\ldots,N$, we then 
have  
\begin{align} \label{estierrory}
\varepsilon_i^{1,y} & \leq \;  C\Big\{ R^2\epsilon^2  + \frac{1+R^2\epsilon^2}{R^{2}} \Big\}, 
\end{align}
for some constant $C(d,\Xc_0)$ independent of $N,R,\epsilon$. Similarly, repeating \eqref{Chebichev} and \eqref{eq: lipschitz zero} by replacing respectively $\widehat{\Uc}^{(1)}_{i}$  by $\widehat{\Zc}_{i}^{(1)}$ and $v_i^{(1)}$ by $ \hat z_i^{(1)}$ and recalling that $\sqrt{\Delta t_k} \widehat{z_{k}}^{(1)},\sqrt{\Delta t_k} \Zc_k^{(1)}$ are $d K$ Lipschitz uniformly w.r.t $N$, we obtain
\begin{align} \label{estierrorz}
\Delta t_i \varepsilon_i^{1,z} & \leq \;  C\{ R^2\epsilon^2  + \frac{1+R^2\epsilon^2}{R^{2}}\}, 
\end{align} Then to obtain a convergence rate of $O(1/N)$ in \eqref{estimMDBDP}, it suffices to choose $R,\epsilon$ such that
\begin{equation*} 
 N R^2\epsilon^2  = O(1/N),\  N \frac{1+R^2\epsilon^2}{R^{2}} = O(1/N), 
\end{equation*} which is verified with if $d>1$ with $R = O(N)$, $\epsilon = O(\frac{1}{N^{2}})$. Then by Proposition \ref{prop: tsb}, we can choose the previously GroupSort neural networks with grouping size $\kappa=O(2\sqrt{d}N^{2})$, depth $\ell + 1 = O(d^2)$ and width $\sum_{i=0}^{\ell-1} m_i = O((2\sqrt{d} N^{2})^{d^2-1})$ if $d>1$. If $d=1$, we can take $\kappa=O(N^{2})$, depth $\ell + 1 = 3$ and width $\sum_{i=0}^{\ell-1} m_i =O(N^{2}) $.

\subsection{Proof of Theorem \ref{theo: CV DS}}

Let us introduce the  {\it explicit} backward Euler scheme of the BSDE \eqref{BSDEfeyn}:
\begin{equation} \label{defbarViexpli} 
\left\{
\begin{array}{ccl}
\bar V_{i}^{(2)} &= &  \E_i\Big[ \bar{V}_{i+1}^{(2)} - f\big(t_{i},X_i,\bar V_{i+1}^{(2)},\bar Z_i^{(2)}\big)\Delta t_i\Big]\\
\bar Z_{i}^{(2)} &= & \E_i\Big[\bar{V}_{i+1}^{(2)} \frac{\Delta W_i}{\Delta t_i} \Big], \quad i=0,\ldots,N-1,  
\end{array} 
\right.
\end{equation}
starting from $\bar V_N^{(2)}$ $=$ $g(X_N)$, and which is also known to converge with the same time discretization error \eqref{eq: time discretization error} than the implicit backward scheme.

\vspace{1mm}

We decompose the approximation error into three terms:
\begin{align}\label{eq: error decomposition 2}
\E\big|Y_{t_i} - \widehat{\Uc}_{i}^{(2)}(X_{i})\big|^2 & \leq \;  3\Big( \E\big|Y_{t_i} -  \bar V_{i}^{(2)} \big|^2 + \E\big| \bar V_{i}^{(2)} - V_{i}^{(2)}\big|^2 + \E\big| V_{i}^{(2)} -\widehat{\Uc}_{i}^{(2)}(X_{i})\big|^2\Big).
\end{align} 
The first term is the classical time discretization error,  and the rest of the proof is devoted to the analysis of the second and third terms.

\vspace{1mm}

\noindent \textit{\underline{Step 1.}}  Fix $i$ $\in$ $\llbracket0,N-1\rrbracket$. 
By definition of $V_i^{(2)}$ in \eqref{defViDS} and the martingale representation theorem, there exists a square integrable process $\{\widehat{Z}_s^{(2)}, t_i\leq s\leq t_{i+1}\}$ such that
\begin{align*}
& \widehat{\Uc}^{(2)}_{i+1}(X_{i+1})  - f\big(t_{i},X_i,\E_i\big[\widehat{\Uc}^{(2)}_{i+1}(X_{i+1})\big],\E_i\big[\sigma(t_i,X_i)\trans D_x\widehat{\Uc}_{i+1}^{(2)}(X_{i+1})\big]\big) \Delta t_i \\ &= \; 
V_i + \int_{t_i}^{t_{i+1}} \widehat{Z}_s^{(2)}. \di W_s.
\end{align*}
It follows that the quadratic loss function of the DS scheme in \eqref{lossDS}  is written as 
\begin{align} 
& J_i^S(\Uc_i) \nonumber 
\\ & := \;  \E\Big|\widehat{\Uc}_{i+1}^{(2)}(X_{i+1}) - \Uc_{i}(X_i) -  f\big(t_i,X_{i+1},\widehat{\Uc}_{i+1}^{(2)}(X_{i+1}), \sigma(t_i,X_i)\trans D_x\widehat{\Uc}_{i+1}^{(2)}(X_{i+1})\big)\Delta t_i\Big|^2 \nonumber \\ 
& = \;  \tilde J_i^S(\Uc_i)  + \E \Big[ \int_{t_i}^{t_{i+1}} |\widehat{Z}_s^{(2)}|^2_{_2} \di s \Big],  \label{JtildeJS}
\end{align}
where 
\begin{align*}
\tilde J_i^S(\Uc_i)  &:= \E\Big| V_i^{(2)} - \Uc_{i}(X_i) + \Delta f_i \Delta t_i \Big|^2 \\
\mbox{ with } \;  \; \Delta f_i & := \; f\big(t_{i},X_i,\E_i[\widehat{\Uc}_{i+1}^{(2)}(X_{i+1})],\E_i[\sigma(t_i,X_i)\trans D_x\widehat{\Uc}_{i+1}^{(2)}(X_{i+1})]\big)  \\
& \quad -  \; f\big(t_i,X_{i+1},\widehat{\Uc}_{i+1}^{(2)}(X_{i+1}), \sigma(t_i,X_i)\trans D_x\widehat{\Uc}_{i+1}^{(2)}(X_{i+1})\big).  
\end{align*}
A direct application of the Young inequality in the form $(a+b)^2 \geq \frac{1}{2} a^2 - b^2$ leads to
\begin{align} \label{tildeJbound}
\tilde J_i^S(\Uc_i)  +   |\Delta t_i|^2  \E\big|\Delta f_i\big|^2    &\geq \;  \frac{1}{2}\E\big| V_i^{(2)} - \Uc_{i}(X_i)\big|^2. 
\end{align}
On the other hand,  by Lipschitz continuity of $f$,  we have 
\begin{align}
& \tilde J_i^S(\Uc_i)  + |\Delta t_i|^2  \E\big|\Delta f_i\big|^2 \nonumber  \\ & \leq \;   2 \E\big| V_i^{(2)} - \Uc_{i}(X_i)\big|^2 + 3  |\Delta t_i|^2  \E\big|\Delta f_i\big|^2 \nonumber  \\
& \leq  \; 2 \E\big| V_i^{(2)} - \Uc_{i}(X_i)\big|^2 +  9 |\Delta t_i|^2  [f]_{_L}^2 \E|X_{i+1} - X_i|_{_2}^2  \nonumber \\
& \;\;\; + \;   9  |\Delta t_i|^2  [f]_{_L}^2 \E\Big|\widehat{\Uc}_{i+1}^{(2)}(X_{i+1}) - \E_i[\widehat{\Uc}_{i+1}^{(2)}(X_{i+1})] \Big|^2  \nonumber \\ 
& \;\;\;  + \; 9  |\Delta t_i|^2 [f]_{_L}^2  \E\Big|\sigma(t_i,X_i)\trans D_x\widehat{\Uc}_{i+1}^{(2)}(X_{i+1})  - \E_i\big[\sigma(t_i,X_i)\trans D_x\widehat{\Uc}_{i+1}^{(2)}(X_{i+1})\big]\Big|_{_2}^2 \nonumber  \\
& \leq  \; 2\ \E\big| V_i^{(2)} - \Uc_{i}^{(2)}(X_i)\big|^2 +  9 |\Delta t_i|^2  [f]_{_L}^2 \E|X_{i+1} - X_i|_{_2}^2 \nonumber \\
&  \;\;\; + \; 9  |\Delta t_i|^2  [f]_{_L}^2 \E\Big|\widehat{\Uc}_{i+1}^{(2)}(X_{i+1}) - \widehat{\Uc}_{i+1}^{(2)}(X_{i}) \Big|^2 \nonumber \\
& \;\;\;  + \; 9  |\Delta t_i|^2 [f]_{_L}^2   \E\Big[ |\sigma(t_i,X_i)|_{_2}^2 \E_i \big| D_x\widehat{\Uc}_{i+1}^{(2)}(X_{i+1})  -  D_x\widehat{\Uc}_{i+1}^{(2)}(X_{i}) \big|_{_2}^2 \Big],  \label{Jintertilde} 
\end{align}
where we use the definition of conditional expectation $\E_i[.]$, and the tower property of conditional expectation in the last inequality. 
Recall  that  $\widehat{\Uc}_{i+1}$ $\in$ $\Nc_{i}^{\gamma,\eta}$ is Lipschitz on $\R^d$. %from Lemma \ref{lemBach}
Actually,  we have  
\begin{align} \label{lipUc} 
\big| \widehat{\Uc}_{i+1}(x) -  \widehat{\Uc}_{i+1}(x') \big| & \leq \; \gamma_i |x-x'|_{_2}, \quad \forall x,x' \in \R^d. 
\end{align}
By the Cauchy-Schwarz inequality, we then have 
\begin{align*}
\E\Big|  \widehat{\Uc}_{i+1}^{(2)}(X_{i+1}) - \widehat{\Uc}_{i+1}^{(2)}(X_{i}^{}) \Big|^2 
& \leq \; 
C\gamma_i^2 \big\| X_{i+1} -   X_{i}^{} \big\|^2_{_4} \\
& \leq \;  C\gamma_i^2 \Delta t_i 
\end{align*} 
for $\Delta t_i$ small enough,  $R$ $\geq$ $1$, and we used again  the standard estimate: 
$\|X_{i}\|_{_{2p}}$ $\leq$ $C(1 + \|\Xc_0\|_{_{2p}})$,  $\|X_{i+1}- X_{i}\|_{_{2p}}$ $\leq$ $C(1 + \|\Xc_0\|_{_{2p}})\sqrt{\Delta t_i}$, for $p$ $\geq$ $1$. 
By using also the Lipschitz condition on  
$D_x \widehat{\Uc}_{i+1}$, and plugging into \eqref{Jintertilde}, we  get 

\begin{align}
\tilde J_i^S(\Uc_i)  + |\Delta t_i|^2  \E\big|\Delta f_i\big|^2  & \leq \;   2  \E\big| V_i^{(2)} - \Uc_{i}(X_i)\big|^2 +  
C(d) \max\big[\gamma_i^2,\eta_i^2\big] \Big( 1 +  \big\|\Xc_0\big\|^2_{_4}  \Big)^2 |\Delta t_i|^3. \label{interJtilde2}
\end{align} 
By applying inequality \eqref{tildeJbound} to $\Uc_i$ $=$ $\widehat{\Uc}_i^{(2)}$, which  is a minimizer of $\tilde{\Jc}_i^{S}$ by \eqref{JtildeJS}, and combining with  
\eqref{interJtilde2}, this yields  for all functions $\Uc_{i}$ in $\Nc_{i}^{\gamma,\eta}$: 
\begin{align*}
\E\big|V_i^{(2)} - \widehat{\Uc}_{i}^{(2)}(X_i) \big|^2 & \leq \;   C\Big(\E\big|V_i^{(2)} - \Uc_{i}(X_i) \big|^2 + (1+\|\Xc_0\|_{_4}^2)^2 |\Delta t_i|^3   \max\big[\gamma_i^2,\eta_i^2\big]  \Big),
\end{align*} 
and thus by minimizing over $\Uc_{i}$ in the right hand side
\begin{align}\label{eq: NN error 2}
\E\big|V_i^{(2)} - \widehat{\Uc}_{i}^{(2)}(X_{i}) \big|^2 &\leq \;  C\Big( \varepsilon_i^{\gamma,\eta} + (1+\|\Xc_0\|_{_4}^2)^2 |\Delta t_i|^3   \max\big[\gamma_i^2,\eta_i^2\big]    \Big).
\end{align}

\vspace{1mm}

\noindent \textit{\underline{Step 2.}} 

From the expressions of $V_i^{(2)}$, and $\bar V_{i}^{(2)}$ in \eqref{defViDS} and \eqref{defbarViexpli}, and by applying the Young, the Cauchy-Schwarz inequalities,  we get with $\beta$ $\in$ $(0,1)$

\begin{align}
& \E\big| \bar V_{i}^{(2)} - V_{i}^{(2)}\big|^2 \nonumber \\ & \leq \; (1+\beta)  \E\Big| \E_i\big[ \widehat{\Uc}^{(2)}_{i+1}(X_{i+1}) - \bar V_{i+1}^{(2)}\big]\Big|^2  \nonumber \\ 
& \; + \;  \big(1+\frac{1}{\beta}\big) |\Delta t_i|^2 \E \Big| f\big(t_{i},X_i,\bar V_{i+1}^{(2)},\bar Z_{i}^{(2)}\big) \nonumber \\ 
&  \hspace{3cm}  - \; f\big(t_{i},X_i,\E_i[\widehat{\Uc}_{i+1}^{(2)}(X_{i+1})],\E_i[\sigma(t_i,X_i)\trans D_x\widehat{\Uc}_{i+1}^{(2)}(X_{i+1})]\big)  \Big|^2 \nonumber \\
&\leq \; (1+\beta)  \E\Big| \E_i\big[ \widehat{\Uc}_{i+1}^{(2)}(X_{i+1}) - \bar V_{i+1}^{(2)} \big]\Big|^2 \nonumber \\
&  \; + \;  2[f]_{_L}^2 \big(1+\frac{1}{\beta}\big)|\Delta t_i|^2\Big( \E\big|\widehat{\Uc}_{i+1}^{(2)}(X_{i+1}) -  \bar V_{i+1}^{(2)}\big|^2  +  
\E\Big|\E_i\big[\sigma(t_i,X_i)\trans D_x\widehat{\Uc}_{i+1}^{(2)}(X_{i+1})\big] - \bar Z_i^{(2)} \Big|_{_2}^2\Big). \label{interDS} 
\end{align}
Now, recalling the expression of $\bar Z_i$ in \eqref{defbarViexpli}, and by a standard integration by parts argument (see e.g. Lemma 2.1 in \cite{FTW11}), we have 
\begin{align*}
&\E_i\big[\sigma(t_i,X_i)\trans D_x\widehat{\Uc}_{i+1}^{(2)}(X_{i+1})\big] - \bar Z_i^{(2)} \\ &= \;  \E_i\Big[ \big( \widehat{\Uc}_{i+1}^{(2)}(X_{i+1}) - \bar V_{i+1}^{(2)} \big)\frac{\Delta W_i}{\Delta t_i} \Big] \\
&= \; \E_i\Big[ \Big( \widehat{\Uc}_{i+1}^{(2)}(X_{i+1}) - \bar V_{i+1}^{(2)}  - \E_i\big[ \widehat{\Uc}_{i+1}^{(2)}(X_{i+1}) - \bar V_{i+1}^{(2)} \big]   \Big)\frac{\Delta W_i}{\Delta t_i} \Big].  
\end{align*}
By plugging into \eqref{interDS}, we then obtain by the Cauchy-Schwarz inequality  
\begin{align}
& \E\big| \bar{V_{i}}^{(2)} - V_{i}^{(2)}\big|^2 \nonumber  \\  &\leq (1+\beta) \E\Big| \E_i\big[ \widehat{\Uc}_{i+1}^{(2)}(X_{i+1}) - \bar V_{i+1}^{(2)}\big]\Big|^2 + 2[f]_{_L}^2
( 1+ \beta) \frac{|\Delta t_i|^2}{\beta}\Big\{  \E\big|\widehat{\Uc}_{i+1}^{(2)}(X_{i+1}) -  \bar V_{i+1}^{(2)}\big|^2 \nonumber \\ 
& \;  + \; \frac{d}{\Delta t_i} \Big[  \E\big|\widehat{\Uc}_{i+1}^{(2)}(X_{i+1}) - \bar V_{i+1}^{(2)} \big|^2 - \E\Big| \E_i\big[\widehat{\Uc}_{i+1}^{(2)}(X_{i+1}) - \bar V_{i+1}^{(2)}\big]\Big|^2\Big] \Big\} \nonumber \\
&\leq  \; (1 + C \Delta t_i)  \E\big|\widehat{\Uc}_{i+1}^{(2)}(X_{i+1}) - \bar V_{i+1}^{(2)} \big|^2,  \label{inter3} 
\end{align} 
by choosing explicitly $\beta$ $=$ $2d[f]_{_L}^2 \Delta t_i$ for $\Delta t_i$ small enough.  By using again the Young inequality on the r.h.s. of \eqref{inter3}, and since $\Delta t_i$ $=$ $O(1/N)$, we then  get
\begin{align*}
\E\big| \bar V_{i}^{(2)} - V_{i}^{(2)}\big|^2  &\leq \; (1 + C \Delta t_i)  \E\big| \bar V_{i+1}^{(2)}  - V_{i+1}^{(2)} \big|^2 +  CN \E\big| \widehat{\Uc}_{i+1}^{(2)}(X_{i+1}) -  V_{i+1}^{(2)} \big|^2. 
\end{align*} 
By discrete Gronwall lemma, and recalling that $\bar V_N^{(2)}$ $=$ $g(X_N)$, $V_N^{(2)}$ $=$ $\widehat{\Uc}_N(X_N)$, we deduce with \eqref{eq: NN error 2} that 
\begin{align}
\sup_{i\in\llbracket 0,N \rrbracket}  \E\big|\bar V_{i}^{(2)}  -  V_{i}^{(2)} \big|^2 & \leq  \; C \varepsilon_N^{\gamma,\eta}  
+ CN \sum_{i=1}^{N-1}  \Big( \varepsilon_i^{\gamma,\eta} + (1+\|\Xc_0\|_{_4}^2)^2 |\Delta t_i|^3   \max\big[\gamma_i^2,\eta_i^2\big]  \Big). \label{Vinter3DS} 
\end{align} 

The required  bound \eqref{errorDS} for the approximation error on $Y$ follows by plugging \eqref{eq: time discretization error}, \eqref{eq: NN error 2} and \eqref{Vinter3DS}  into \eqref{eq: error decomposition 2}. 
\ep

\subsection{Proof of Proposition \ref{cor: convergence DBDP}}
For $x$ $\in$ $\R^d$, we define the processes  $X_{j+1}^{j,x}$, $j$ $=$ $0,\ldots,N$, 
\begin{align*}
X_{j+1}^{j,x} &:=  x  + \mu(t_j,x) \Delta t_j + \sigma(t_j,x) \Delta W_j, \quad j=0,\ldots,N-1. 
\end{align*} Define also
\begin{equation*}  
\left\{
\begin{array}{ccl}
V_{i,3}^{x} & = & \E_i\Big[\widehat{\Uc}_{i+1}^{(3)}(X_{i+1}^{i,x})-  f\big(t_{i},x,V_{i,3}^{x},\overline{\widehat{Z}}_{i,3}^{x}\big)\Delta t_i \Big] = v_i^{(3)}(x)\\
\overline{\widehat{Z}}_{i,3}^{x} & = &  \E_i\Big[\widehat{\Uc}_{i+1}^{(3)}(X_{i+1}^{i,x}) \frac{\Delta W_i}{\Delta t_i}\Big] = \widehat{z}_{i}^{(3)}(x)
\end{array}
\right.
\end{equation*} with $v_i^{(3)},\widehat{z}_{i}^{(3)}$ as in \eqref{eq : Markov functions 3} by Markov property.\\
\noindent \textit{\underline{Step 1.}} 
Let $x'\in\R^d$. By the Cauchy-Schwarz inequality, we have the standard estimate 
\begin{align}\label{eq: Z estimation}
   & \Delta t_i \E\Big|\overline{\widehat{Z}}_{i,3}^{x} - \overline{\widehat{Z}}_{i,3}^{x'} \Big|_2^2 \nonumber \\ & = \frac{1}{\Delta t_i} \E\Big|\E_i\Big[\{\widehat{\Uc}_{i+1}^{(3)}(X_{i+1}^{i,x}) - \widehat{\Uc}_{i+1}^{(3)}(X_{i+1}^{i,x'})- \E_i\Big[\widehat{\Uc}_{i+1}^{(3)}(X_{i+1}^{i,x}) - \widehat{\Uc}_{i+1}^{(3)}(X_{i+1}^{i,x'})\Big]\}  \Delta W_i \Big]\Big|^2\nonumber \\
    & \leq d\ \Big(\E\Big|\widehat{\Uc}_{i+1}^{(3)}(X_{i+1}^{i,x}) - \widehat{\Uc}_{i+1}^{(3)}(X_{i+1}^{i,x'})\Big|^2 - \E\Big| \E_i\Big[\widehat{\Uc}^{(3)}_{i+1}(X_{i+1}^{i,x}) - \widehat{\Uc}_{i+1}^{(3)}(X_{i+1}^{i,x'})\Big]\Big|^2\Big).
\end{align} 
We then apply the Young inequality to see that
\begin{align*}
    & \E\Big|V_{i,3}^{x} - V_{i,3}^{x'} \Big|^2 \\& \leq (1+\gamma\Delta t_i) \E\Big|\E_i\Big[\widehat{\Uc}_{i+1}^{(3)}(X_{i+1}^{i,x})-\widehat{\Uc}_{i+1}^{(3)}(X_{i+1}^{i,x'})\Big]\Big|^2 \\ & + (1+\frac{1}{\gamma\Delta t_i})\Delta t_i^2\E\Big|\{ f\big(t_{i},x',V_{i,3}^{x'},\overline{\widehat{Z}_{i}}^{3,x'}\big)-f\big(t_{i},x,V_{i,3}^{x},\overline{\widehat{Z}}_{i,3}^{x}\big)  \}  \Big|^2\\%%%%%%%%%%%%%%
    & \leq (1+\gamma\Delta t_i) \E\Big|\E_i\Big[\widehat{\Uc}_{i+1}^{(3)}(X_{i+1}^{i,x})-\widehat{\Uc}_{i+1}^{(3)}(X_{i+1}^{i,x'})\Big]\Big|^2 \\ & + 3[f]^2(1+\frac{1}{\gamma\Delta t_i})\Delta t_i^2\E\{|x-x'|_2^2 + |V_{i,3}^{x}-V_{i,3}^{x'}|^2 + |\overline{\widehat{Z}}_{i,3}^{x}-\overline{\widehat{Z}}_{i,3}^{x'}|_2^2  \}.  
\end{align*} Hence for $\gamma = 3[f]^2d$  and $\Delta t_i$ small enough, using \eqref{eq: Z estimation} we obtain %\blue{Specify use of equation of Z above}
\begin{align*}
      \E\Big|V_{i,3}^{x} - V_{i,3}^{x'} \Big|^2  & \leq (1+(\gamma+3d)\Delta t_i) \E\Big|\widehat{\Uc}_{i+1}^{(3)}(X_{i+1}^{i,x})-\widehat{\Uc}_{i+1}^{(3)}(X_{i+1}^{i,x'})\Big|^2 \\ & + (1+(\gamma+3d)\Delta t_i)\Delta t_i\E|x-x'|_2^2.
\end{align*} Therefore, with Lemma \ref{lem: euler Lipschitz} \begin{align*}
     |v_{N-1}^{(3)}(x) - v_{N-1}^{(3)}(x')|^2 & = \E\Big|V_{N-1,3}^{x} - V_{N-1,3}^{x'} \Big|^2\\  & \leq (1+(\gamma+3d)\Delta t_{N-1})((1+C\Delta t_{N-1}) [g]^2 + \Delta t_i) |x-x'|_2^2
     \\  & \leq (1+\hat C\Delta t_{N-1}) ([g]^2 + \Delta t_i) |x-x'|_2^2,
\end{align*} for some constant $\hat C$. Similarly, assuming $\widehat{\Uc}_{i+1}^{(3)}$ is $[\widehat{\Uc}_{i+1}^{(3)}]-$Lipschitz, $v_{i}^{(3)}$ is Lipschitz with constant $[v_i^{(3)}]$ verifying
\begin{align*}
     [v_i^{(3)}]^2 \leq (1+\hat C\Delta t_i) ([\widehat{\Uc}_{i+1}^{(3)}]^2 + \Delta t_i).
\end{align*}
\noindent \textit{\underline{Step 2.}} Let $\epsilon>0,\ \kappa\in\N,\ \ell\in\N,\ m\in\R^\ell$ to be chosen after. Recursively, we approximate $ v_{i}^{(3)}$ by a $[v_i^{(3)}]$-Lipschitz GroupSort neural network $\Uc_i^{(3)}$ in $\Nc_i = \Gc_{[v_i^{(3)}],d,1,\ell,m}^{\zeta_\kappa}$ with uniform error $2[v_i]R\epsilon$ on $[-R,R]^d$ (Proposition \ref{prop: tsb}).
Then by discrete Gronwall inequality\begin{align*}
[\Uc_i^{(3)}]^2    \leq K,\
     [v_i^{(3)}]^2     \leq K,
\end{align*} uniformly in $i, N$ for some constant $K$. 
Thus $v_i^{(3)},\Uc_i^{(3)}$ are $K$ Lipschitz, uniformly. Then we approximate by \eqref{eq: Z estimation} $  \widehat{z_{i}}^{(3)}$ by a $\sqrt{\frac{d}{\Delta t_i}}[v_i^{(3)}]$-Lipschitz GroupSort  neural network $\Zc_i$ in $\Nc_i'=\Gc_{\sqrt{\frac{d}{\Delta t_i}}[v_i^{(3)}],d,d,\ell,m}^{\zeta_\kappa}$ with uniform error $2\frac{d}{\sqrt{\Delta t_i}}[v_i^{(3)}]R\varepsilon$ on $[-R,R]^d$ thanks to Proposition \ref{prop: tsb}. Thus $\sqrt{\Delta t_i} \widehat{z_{i}}^{(3)},\sqrt{\Delta t_i} \Zc_i^{(3)}$ are $d K$ Lipschitz, uniformly.\\
\noindent \textit{\underline{Step 3.}} The regression errors $\varepsilon_i^{3,y}$ verify from, localization of $X_i$ on  $B_2(R)$,  
%$=$ $\{x\in\R^d: |x|_{_2}\leq R\}$,   
the H\"older inequality, and  the Markov inequality,  the approximation error of $v_i^{(3)}$,  $i\in\llbracket0,N-1\rrbracket$,  by the class of GroupSort neural networks (Proposition \ref{prop: tsb})
\begin{align}
\sqrt{ \varepsilon_i^{3,y} } & = \;  \inf_{\Uc_{}\in\Gc_{[v_k],d,1}} \big\| v_i^{(3)}(X_i) - \Uc_{}(X_i)\big\|_{_2}  \nonumber \\
& \leq \;  \inf_{\Uc_{}\in\Gc_{[v_k],d,1}}  \Big\| \big(v_i^{(3)}(X_i) - \Uc(X_i)\big) {\bf 1}_{_{X_i \in B_2(R)}} \Big\|_{_2}  \; +   
\Big\| \big(v_i^{(3)}(X_i) -  \widehat{\Uc}_{i}^{(3)}(X_i)\big) {\bf 1}_{_{|X_i|_{_2}\geq R}} \Big\|_{_2}   \nonumber \\
& \leq \;  2KR\epsilon  + \;   \E\Big|\big(v_i^{(3)}(X_i) -  \widehat{\Uc}_{i}^{(3)}(X_i)\big)^{2q}\Big|^{1/{2q}} \E\Big|{\bf 1}_{_{|X_i|_{_2}\geq R}}^\frac{2q}{2q-1} \Big|^{\frac{2q-1}{2q}} \nonumber\\
& = \;  2KR\epsilon  + \;\E\Big|\big(v_i^{(3)}(X_i) -  \widehat{\Uc}_{i}^{(3)}(X_i)\big)^{2q}\Big|^{1/{2q}} \E[{\bf 1}_{_{|X_i|_{_2}\geq R}}]^{\frac{2q-1}{2q}} \nonumber\\
& \leq \;  2KR\epsilon  + \;   \frac{ \big( \big\|v_i^{(3)}(X_i)-v_i^{(3)}(0)\|_{_{2q}}  + \big\|\widehat{\Uc}_{i}^{(3)}(X_i)-v_i^{(3)}(0)\big\|_{_{2q}} \big) \big\|X_i\big\|_{_\frac{2 q}{2q-1}} }{R},  \label{Chebichev DBDP} 
\end{align} 
by noticing that $(v_i^{(3)}(X_i) -  \widehat{\Uc}_{i}^{(3)}(X_i)\big)=(v_i^{(3)}(X_i)-v_i^{(3)}(0) -  (\widehat{\Uc}_{i}^{(3)}(X_i)-v_i^{(3)}(0))\big)$ for $q$ $>$ $0$ and $2q =  2+\delta$ with $\delta$ as in the statement of the Proposition.  
Now, by Lipschitz continuity of $v_i^{(3)},\widehat{\Uc}^{(3)}$ and because $0\in B_2(R)$ we have 
\begin{align}
& \big\|\widehat{\Uc}_{i}^{(3)}(X_i)-v_i^{(3)}(0)\big\|_{_{2q}} + \big\|v_{i}^{(3)}(X_i)-v_i^{(3)}(0)\big\|_{_{2q}} \\ &\leq \big\|\widehat{\Uc}_{i}^{(3)}(0)-v_i^{(3)}(0)\big\|_{_{2q}} + \big\|\widehat{\Uc}_{i}^{(3)}(X_i)-\widehat{\Uc}_{i}^{(3)}(0)\big\|_{_{2q}} +  \big\|v_{i}^{(3)}(X_i)-v_i^{(3)}(0)\big\|_{_{2q}}\\ & \leq 2KR\epsilon  + 2 K \big\|X_i\big\|_{_{2q}}.\label{eq: lipschitz zero DBDP}
\end{align}  
Recalling the standard estimate $\|X_i\|_{_{2q}}$ $\leq$ $C(1+\|\Xc_0\|_{_{2q}})$, $i$ $=$ $0,\ldots,N$, we then 
have  
\begin{align} \label{estierrory DBDP}
\varepsilon_i^{3,y} & \leq \;  C\Big\{ R^2\epsilon^2  + \frac{1+R^2\epsilon^2}{R^{2}} \Big\}, 
\end{align}
for some constant $C(d,\Xc_0)$ independent of $N,R,\epsilon$. 
Similarly repeating \eqref{Chebichev DBDP} and \eqref{eq: lipschitz zero DBDP} by replacing respectively $\widehat{\Uc}^{(3)}_{i}$  by $\widehat{\Zc}_{i}^{(3)}$ and $v_i^{(3)}$ by $ \hat z_i^{(3)}$ and recalling that $\sqrt{\Delta t_i} \widehat{z_{i}}^{(3)},\sqrt{\Delta t_i} \Zc_i^{(3)}$ are $d K$ Lipschitz uniformly w.r.t. $N$, we obtain
\begin{align} \label{estierrorz DBDP}
\Delta t_i \varepsilon_i^{3,z} & \leq \;  C\Big\{ R^2\epsilon^2  + \frac{1+R^2\epsilon^2}{R^{2}} \Big\}. 
\end{align} 
Then to obtain a convergence rate of $O(1/N)$ in \eqref{estimDBDP}, it suffices to choose $R,\varepsilon$ such that
\begin{equation*} 
 N^2 R^2\epsilon^2  = O(1/N), \quad   N^2 \frac{1+R^2\epsilon^2}{R^{2}} = O(1/N),
\end{equation*} 
which is verified with $R = O(N^{3/2})$, $\epsilon = O(\frac{1}{N^{3}})$. Then by Proposition \ref{prop: tsb}, if $d>1$ we can choose the previously GroupSort neural networks with grouping size $\kappa=O(\lceil 2\sqrt{d} N^{3}\rceil)$, depth $\ell + 1 = O(d^2)$ and width $\sum_{i=0}^{\ell-1} m_i = O((2\sqrt{d} N^{3})^{d^2-1})$. %(1+1/p)/2 
If $d=1$, we can take $\kappa=O( N^{3})$, depth $\ell + 1 = 3$ and width $\sum_{i=0}^{\ell-1} m_i =O(N^{3}) $.

\section{Numerical Tests}\label{sec: numerics}

We test  our different algorithms and the cited ones in this paper on some examples and by varying  the state space dimension. 
In each example we use tanh as activation function, and an architecture composed of 2 hidden layers with $d+10$ neurons. We apply Adam gradient descent \cite{KB14} with a decreasing learning rate, using the Tensorflow library. 
Each numerical experiment is conducted using a node composed of 2 Intel® Xeon® Gold 5122 Processors, 192 Gb of RAM, and 2 GPU nVidia® Tesla® V100 16Gb.  We use a batch size of 1000. 
We do not implement the GroupSort network because even if  it is useful for theoretical analysis, it would be costly to use in practice: on the one hand, it 
will induce a cost of order $O(n\ln n)$ where $n$ is the batch size, compared to a linear cost $O(n)$  for standard activation function; on the other hand, it requires  to track the Lipschitz constant of the functions and adapt the networks architecture accordingly. Whereas theoretical results suggest to take deep neural networks with depth increasing with the dimension, we  observe that two hidden layers are enough to obtain a good accuracy. According to our experience $\tanh$ activation function  provides  the best results. ReLU or Elu being not bounded, some explosion tends to appear when the learning rates are not small enough.\\
We consider examples from \cite{HPW19} to compare its DBDP scheme with the DS and MDBDP schemes.  
The three first lines of the tables below are taken from \cite{HPW19}. For each test, the two best results are highlighted in boldface. We use 5000 gradient descent iterations by time step except 20000 for the projection of the final condition. The execution of the multistep algorithm approximately takes between 
8000 s. and 16000 s. (depending on the dimension) for a resolution with $N = 120$. More numerical examples and tests are presented  
in the extended version \cite{GPW20} of this paper, and the codes
at: \url{https://github.com/MaxGermain/MultistepBSDE}.

\subsection{PDE with Bounded Solution and Simple Structure}\label{sec: Pde simple}

We take the parameters: $\mu$ $=$ $\frac{0.2}{d}$,  $\sigma$ $=$ $\frac{I_d}{\sqrt{d}}$, terminal condition $g(x)$ $=$ $\cos(\overline{x})$, with $\overline{x} = \sum_{i=1}^d x_i$, and generator
\begin{eqnarray}
& & \quad f(x,y,z) \label{PDE semi BOUNDED} \\
& =  &  - \Big(\cos(\overline{x})+0.2\sin(\overline{x})\Big)e^{\frac{T-t}{2}} + \frac{1}{2} (\sin(\overline{x})\cos(\overline{x})e^{T-t})^2  \nonumber 
- \frac{1}{2d}(y(1_d \cdot z))^2.
\end{eqnarray}
so  that the PDE solution is given by
$u(t,x)$ $=$ $\cos\left(\overline{x}\right) \exp\left(\frac{T-t}{2}\right)$. 

We fix $T$ $=$ $1$, and increase the dimension $d$. The results are reported 
%in Figure \ref{fig: table results bounded d=5} for $d$ $=$ $5$,  
in  Figure \ref{fig: table results bounded d=10} for $d$ $=$ $10$,  in Figure \ref{fig: table results bounded d=20} for $d$ $=$ $20$, and in Figure \ref{fig: table results bounded d=50} for $d$ $=$ $50$. It is observed that all the schemes DBDP, DBSDE and MDBDP provide quite accurate results 
%with comparable precision, 
with smallest standard deviation for MDBDP, and largely outperforms the DS scheme.

\begin{small}

\begin{figure}[H]
	\centering
	\begin{tabular}{|c|c|c|c|}
		\hline
		& Averaged value & Standard deviation & Relative error (\%)\\
		\hline
		[HPW20] (DBDP1) & - 1.3895 & 0.0015 & 0.44\\
		\hline
		[HPW20] (DBDP2)  & - 1.3913 & \textbf{0.0006} & 0.57\\
		\hline
		[HJE17] (DBSDE) & \textbf{- 1.3880} & 0.0016 & \textbf{0.33}\\
		\hline
		[Bec+19] (DS)  & - 1.4097  &  0.0173 & 1.90\\
		\hline 
		MDBDP  & \textbf{-1.3887} & \textbf{0.0006} & \textbf{0.38}\\
		\hline 
	\end{tabular}
	\caption{Estimate of $u(0, x_0 )$ in the case \eqref{PDE semi BOUNDED}, where $d = 10, x_0 = 1\ \mathds{1}_{10}, T=1$ with 120 time steps. Average and standard deviation
		observed over 10 independent runs are reported. The theoretical solution is -1.383395.}
	\label{fig: table results bounded d=10}
\end{figure}

\end{small}

\begin{figure}[H]
	\centering
	\begin{tabular}{|c|c|c|c|}
		\hline
		& Averaged value & Standard deviation & Relative error (\%)\\
		\hline
		[HPW20] (DBDP1) & 0.6760 & 0.0027 & 0.47\\
		\hline
		[HPW20] (DBDP2)  & \textbf{0.6710} & 0.0056 & \textbf{0.27}\\
		\hline
		[HJE17] (DBSDE) & 0.6869 & \textbf{0.0024} & 2.09\\
		\hline
		[Bec+19] (DS)  & 0.6944  & 0.0201 & 3.21 \\
		\hline 
		MDBDP  & \textbf{0.6744} & \textbf{0.0005} & \textbf{0.24}\\
		\hline 
	\end{tabular}
	\caption{Estimate of $u(0, x_0 )$ in the case \eqref{PDE semi BOUNDED}, where $d = 20, x_0 = 1\ \mathds{1}_{20}, T=1$ with 120 time steps. Average and standard deviation
		observed over 10 independent runs are reported. The theoretical solution is 0.6728135.}
	\label{fig: table results bounded d=20}
\end{figure}

\begin{figure}[H]
	\centering
	\begin{tabular}{|c|c|c|c|}
		\hline
		& Averaged value & Standard deviation & Relative error (\%)\\
		\hline
		[HPW20] (DBDP1) & \textbf{1.5903} & \textbf{0.0063} & \textbf{0.04}\\
		\hline
		[HPW20] (DBDP2)  & 1.5876 & 0.0068 & 0.21\\
		\hline
		[HJE17] (DBSDE) & 1.5830 & 0.0361 & 0.50\\
		\hline
		[Bec+19] (DS)  & 1.6485 & 0.0140 & 3.62\\
		\hline 
		MDBDP  & \textbf{1.5924} & \textbf{0.0005} & \textbf{0.09}\\
		\hline 
	\end{tabular}
	\caption{Estimate of $u(0, x_0 )$ in the case \eqref{PDE semi BOUNDED}, where $d = 50, x_0 = 1\ \mathds{1}_{50}, T=1$ with 120 time steps. Average and standard deviation
		observed over 10 independent runs are reported. The theoretical solution is 1.5909.}
	\label{fig: table results bounded d=50}
\end{figure}

\subsection{PDE with Unbounded Solution and more Complex Structure}

We consider a toy example with solution given by 
\begin{align*}
u(t,x) = \frac{T-t}{d} \sum_{i=1}^d (\sin(x_i) 1_{x_i<0} + x_i 1_{x_i\geq0}) + \cos\Big(\sum_{i=1}^d i x_i\Big).
\end{align*}
Therefore we take the parameters
\begin{align}\label{PDE semi UNBOUNDED}
%\begin{cases}
\mu =0,\ \sigma = \frac{I_d}{\sqrt{d}}, \;\; T = 1, \quad
f(t,x,y,z) = k(t,x) - \frac{y}{\sqrt{d}} (1_d\cdot z) - \frac{y^2}{2}
%\end{cases}
\end{align}
with $k(t,x)  = \partial_t u + \frac{1}{2d}\Tr(D^2_x u) + \frac{u}{\sqrt{d}} \sum_i D_{x_i} u + \frac{u^2}{2}.$

We start with tests in dimension $d$ $=$ $1$. The  results are reported in Figure \ref{fig: table results d=1}.

\begin{figure}[H]
	\centering
	\begin{tabular}{|c|c|c|c|}
		\hline
		& Averaged value & Standard deviation & Relative error (\%)\\
		\hline
		[HPW20] (DBDP1) & 1.3720 & 0.0030 & 0.41\\
		\hline
		[HPW20] (DBDP2)  & \textbf{1.3736} & 0.0022 & \textbf{0.29}\\
		\hline
		[HJE17] (DBSDE) & 1.3724 & \textbf{0.0005} & 0.38\\
		\hline
		[Bec+19] (DS)  & 1.3630 & 0.0079 & 1.06\\
		\hline 
		MDBDP  & \textbf{1.3735}  & \textbf{0.0003} & \textbf{0.30}\\
		\hline 
	\end{tabular}
	\caption{Estimate of $u(0, x_0 )$ in the case \eqref{PDE semi UNBOUNDED}, where $d = 1, x_0 = 0.5$, $T$ $=$ $1$ with 120 time steps. Average and standard deviation
		observed over 10 independent runs are reported. The theoretical solution is 1.3776.}
	\label{fig: table results d=1}
\end{figure}

We next  increase the dimension  to $d$ $=$ $8$, and report the results in the following figure. The accuracy is not so good as in the previous section with simple structure of the solution, but we notice that the MDBDP scheme yields the best performance (above dimension $d$ $=$ $10$, all the schemes do not give good approximation results).

\begin{figure}[H]
	\centering
	\begin{tabular}{|c|c|c|c|}
		\hline
		& Averaged value & Standard deviation & Relative error (\%)\\
		\hline
		[HPW20] (DBDP1) & \textbf{1.1694} & 0.0254 & \textbf{0.78}\\  
		\hline
		[HPW20] (DBDP2) & 1.0758 & \textbf{0.0078} & 7.28\\
		\hline
		[HJE17] (DBSDE) & NC & NC&  NC\\
		\hline
		[Bec+19] (DS)  & 1.2283 & \textbf{0.0113} & 5.86\\
		\hline 
		MDBDP  &  \textbf{1.1654} & 0.0379 & \textbf{0.47}\\
		\hline 
	\end{tabular}
	\caption{Estimate of $u(0, x_0 )$ in the case \eqref{PDE semi UNBOUNDED}, where $d = 8, x_0 = 0.5\ \mathds{1}_8$, $T$ $=$ $1$ with 120 time steps. Average and standard deviation
		observed over 10 independent runs are reported. The theoretical solution is 1.1603.}
	\label{fig: table results d=8}
\end{figure}

\printbibliography

\end{document}